\documentclass[reqno]{amsart}
\usepackage[utf8]{inputenc}
\usepackage[dvips]{graphicx,color}
\usepackage[all]{xy}
\usepackage{enumerate}
\usepackage{graphicx,pstricks,pst-plot}

\usepackage{setspace}
\usepackage[left=3.2cm, right=3.2cm, bottom=4cm]{geometry}

\def\today{\ifcase\month\or
  January\or February\or March\or April\or May\or June\or
  July\or August\or September\or October\or November\or December\fi
  \space\number\day, \number\year}
\DeclareMathOperator{\sgn}{\mathrm{sgn}}

\DeclareMathOperator{\supp}{\mathrm{supp}}

 \newtheorem{theorem}{Theorem}
 \newtheorem{lemma}[theorem]{Lemma}
 \newtheorem{proposition}[theorem]{Proposition}
 \newtheorem{corollary}[theorem]{Corollary}
 \theoremstyle{definition}

 \theoremstyle{remark}
 \newtheorem{remark}[theorem]{Remark}

  \newcommand{\mc}{\mathcal}
 \newcommand{\A}{\mc{A}}

 \newcommand{\E}{\mc{E}}

 \newcommand{\C}{\mathbb{C}}
 \newcommand{\R}{\mathbb{R}}

 \newcommand{\ds}{\text{\rm d}s}

 \newcommand{\dt}{\text{\rm d}t}
  \renewcommand{\d}{\text{\rm d}}
 \newcommand{\du}{\text{\rm d}u}

 \newcommand{\dw}{\text{\rm d}w}
 \newcommand{\dx}{\text{\rm d}x}
 
 \newcommand{\dy}{\text{\rm d}y}

\newcommand{\im}{{\rm Im}\,}
\newcommand{\re}{{\rm Re}\,}

\begin{document}

\title[Fourier optimization and prime gaps]{Fourier optimization and prime gaps}
\author[Carneiro, Milinovich, and Soundararajan]{Emanuel Carneiro, Micah B. Milinovich, and Kannan Soundararajan}
\address{IMPA - Instituto Nacional de Matem\'{a}tica Pura e Aplicada - Estrada Dona Castorina, 110, Rio de Janeiro, RJ 22460-320, Brazil.}
\email{carneiro@impa.br}
\address{Department of Mathematics, University of Mississippi, University, MS 38677, USA.}
\email{mbmilino@olemiss.edu}
\address{Department of Mathematics, Stanford University, Stanford, CA 94305, USA.}
\email{ksound@stanford.edu}

\date{\today}
\allowdisplaybreaks
\numberwithin{equation}{section}

\maketitle

\begin{abstract}
We investigate some extremal problems in Fourier analysis and their connection to a problem in prime number theory. In particular, we improve the current bounds for the largest possible gap between consecutive primes assuming the Riemann hypothesis.
\end{abstract}

\section{Introduction}

In this paper we study a new set of extremal problems in Fourier analysis, motivated by a problem in prime number theory. 
These problems (which will be described shortly) are of the kind where one prescribes some constraints for a function and its Fourier transform, and then wants to optimize a certain quantity. When available, a solution to such a problem usually requires two main ingredients: a tool to prove optimality and a tool to construct an extremal function. A classical example in approximation theory is the problem of finding the best $L^1(\R)$-approximation of real-valued functions by bandlimited functions (i.e. functions with compactly supported Fourier transforms). For the two-sided problem (i.e. unrestricted approximation), one usually works with the so called {\it extremal signatures} to establish optimality, whereas for the one-sided problem (in which one is interested in majorizing or minorizing a given function) the Poisson summation formula is useful as a tool to prove optimality. For an account of such methods see, for instance, \cite{CLV, Na, V} and the references therein. Optimal bandlimited majorants and minorants have several applications to inequalities in analysis and number theory, for instance in connection to the theory of the Riemann zeta-function, e.g. \cite{CCLM, CCM, CChi, CS}. Slightly different extremal problems appear in the work \cite{LLS}, in connection with the question of bounding the least quadratic nonresidue modulo a prime. Another example of a Fourier optimization problem was proposed by Cohn and Elkies \cite{CE}, in connection to the sphere packing problem. This recently attracted considerable attention with its resolution in dimensions $8$ and $24$ (see \cite{CKMRV, Vi}).

\smallskip

As we see below, the Fourier optimization problems considered here are simple enough to be stated in very accessible terms but rather delicate in the sense that the usual tools in the literature to prove optimality and construct extremal functions are not particularly helpful. 
While we have been unable to determine explicitly the solutions to our optimization problems, we are able to make progress on the existence and uniqueness of extremizers, and to establish good upper and lower bounds for the values of the sharp constants.  
In addition, we establish a connection between these extremal problems in Fourier analysis and the problem of bounding the largest possible gap between 
consecutive primes (assuming the Riemann hypothesis).

\subsection{Fourier optimization problems} For $F \in L^1(\R)$, we let
\[
\widehat{F}(t) = \int_{-\infty}^{\infty} e^{-2 \pi i x t }\,F(x)\,\dx
\]
denote the Fourier transform of $F$. We also let $x_+ := \max\{x,0\}$ and $1\leq A \leq \infty$ be a given parameter (note that we include the possibility that $A=\infty$), 
and we consider the following problems.

\smallskip

\noindent {\it {\bf Extremal problem 1:} Given $1\le A <\infty$, find 
\begin{equation}\label{Extremal_Problem_1}
{\mc C}(A) := \sup_{\substack{F \in \mc{A} \\ F \neq 0}} \frac{1}{\| F\|_1} \Big ( |F(0)| - A\int_{[-1,1]^c} \big|\widehat{F}(t)\big|\,\dt \Big)\,,
\end{equation}
where the supremum is taken over the class $\mathcal{A}$ of continuous functions $F:\R \to \C$, with $F \in L^1(\R)$.
In the case $A=\infty$, determine 
\begin{equation}\label{Extremal_Problem_1_infty}
{\mc C}(\infty) =  \sup_{\substack{F \in \E  \\ F \neq 0}} \frac{|F(0)|}{\| F\|_1}, 
\end{equation}
where the supremum is over the subclass $\E \subset \mathcal{A}$ of continuous functions $F:\R \to \C$, with $F \in L^1(\R)$ and $\supp\big( \widehat{F}\big) \subset [-1,1]$.}
\smallskip

\noindent {\it {\bf Extremal problem 2:} Given $1\le A < \infty$, find 
\begin{equation}\label{Extremal_Problem_2}
{\mc C}^+(A) := \sup_{\substack{\ \ F \in \mc{A}^+ \\ F \neq 0}}  \frac{1}{\|F\|_{1}} \Big( F(0) - A\int_{[-1,1]^c} \big(\widehat{F}(t)\big)_+\,\dt \Big) \,,
\end{equation}
where the supremum is taken over the class $\mathcal{A}^{+}$ of even and continuous functions $F:\R \to \R$, with $F \in L^1(\R)$.
In the case $A=\infty$, determine 
\begin{equation}\label{Extremal_Problem_2_infty}
{\mc C}^+(\infty) =  \sup_{\substack{\ F \in \mc{E}^+ \\ F \neq 0}} \frac{F(0)}{\| F\|_1}, 
\end{equation}
where the supremum is over the subclass $\mc{E}^+ \subset \mathcal{A}^+$ of even and continuous functions $F:\R \to \R$, with $F \in L^1(\R)$ and $\widehat{F}(t) \leq0$ for $|t| \geq 1$.
}

\smallskip

There has been some previous works in connection to problem \eqref{Extremal_Problem_1_infty} and its analogue for trigonometric polynomials, see for instance \cite{AKP, Gor, Tai}. The current best numerical upper and lower bounds for ${\mc C}(\infty)$, reviewed in \eqref{April14_10:41am} below, are due to Gorbachev \cite[Theorem 3]{Gor2}. We were not able to find any mention to the other problems in the literature. {\it If one further imposes the condition that $F$ is nonnegative on $\R$}, then \eqref{Extremal_Problem_1_infty} reduces to a folkloric problem for bandlimited functions while \eqref{Extremal_Problem_2_infty} reduces to the Cohn-Elkies problem \cite[Theorem 3.1]{CE} in dimension $1$. In both cases Poisson summation shows that the required maximum is $1$, being attained by any constant multiple of the Fej\'{e}r kernel $F(x) =  \big(\sin(\pi x)/(\pi x) \big)^2$. Classical interpolation formulas of Vaaler \cite[Theorem 9]{V} show that these are indeed the unique extremizers for this simplified version of \eqref{Extremal_Problem_1_infty}, whereas this simplified version of \eqref{Extremal_Problem_2_infty} admits other extremizers (see \cite[Section 5]{CE}).

\smallskip

We restricted the parameter $A$ to the range $1 \leq A \leq \infty$ because in the range $0<A<1$ the corresponding problems \eqref{Extremal_Problem_1} and \eqref{Extremal_Problem_2} are trivial in the sense that ${\mc C}(A) ={\mc C}^+(A) = \infty$. This can be seen by taking $F_{\varepsilon}(x) = \frac{1}{\sqrt{\varepsilon}  }\,e^{-\pi x^2/\varepsilon}$ with $\varepsilon \to 0^+$. It is also clear that the mappings $A \mapsto {\mc C}(A)$ and $A \mapsto {\mc C}^+(A) $ are non-increasing for $1\leq A \leq \infty$.

\smallskip

The extremal problems presented here are certainly related to the phenomenon of Fourier uncertainty, and works like \cite{DL,DS}, that discuss $L^1$-uncertainty principles, provide interesting insights. The recent works \cite{BKK, GOS} on the ``root-uncertainty principle" for the Fourier transform also consider interesting extremal problems related to the theory of zeta-functions in number fields. Toward the problems of determining the exact values of the sharp constants ${\mc C}(A)$ and ${\mc C}^+(A)$ 
we establish the following results.  

\begin{theorem} \label{Thm1} Let $1 \leq A \leq \infty$. With respect to problems \eqref{Extremal_Problem_1} and \eqref{Extremal_Problem_1_infty}, the following propositions hold:
\begin{itemize}
\item[(a)] If $A = \infty$, then:

\begin{itemize}
\item[(a.1)] There exists an even and real-valued function $G \in \E$, with $G(0)=1$, that extremizes \eqref{Extremal_Problem_1_infty}. 
\item[(a.2)] All the extremizers of \eqref{Extremal_Problem_1_infty} are of the form $F(x) = c\,G(x)$, where $c\in \C$ with $c\ne0$. 
\item[(a.3)] The extremal function $G$ verifies the identity
\begin{equation*}
{\mc C}(\infty) \int_{-\infty}^{\infty} \sgn(G(x)) \,F(x)\,\dx = F(0)
\end{equation*}
for any $F \in \E$.
\item[(a.4)] $($cf. \cite{Gor2}$)$ The sharp constant ${\mc C}(\infty)$ verifies the inequality
\begin{equation}\label{April14_10:41am}
1.08185\ldots \leq  {\mc C}(\infty) \leq  1.09769\ldots.
\end{equation}
\end{itemize}

\item[(b)] If $A=1$ then ${\mc C}(1)=2$, but there are no extremizers for \eqref{Extremal_Problem_1}.

\smallskip

\item[(c)] If $1 < A < \infty$, then:
\begin{itemize}
\item[(c.1)] There exists an even and real-valued function $G \in {\mc A}$ that extremizes \eqref{Extremal_Problem_1}. 
\item[(c.2)] Let $c_0 = \frac{4}{\pi} \left(\int_{-1}^1 \frac{\sin \pi t}{\pi t} \dt \right)^{-1} = 1.07995\ldots $ and $d_0 = 1.09769\ldots$ be the constant on the right-hand side of \eqref{April14_10:41am}. Let $\lambda = \lambda(A)$ be the unique solution of 
$$1 - \frac{1}{A} = \sin\left(\frac{\pi \lambda}{2}\right) - \frac{\pi \lambda}{2}\cos\left(\frac{\pi \lambda}{2}\right)$$
with $0 < \lambda < 1$. Then
\begin{equation}\label{May31_11:57pm}
\max\left\{2A - 2\sqrt{A(A-1)}\ , \  \frac{\pi A\,c_0}{2} \cos\left(\frac{\pi \lambda(A)}{2}\right)\right\}\, \leq\,  {\mc C}(A)\, \leq\, \min\left\{ \left( \frac{d_0}{1 - \frac{0.3}{(A-2)}}\right) ,\,2  \right\},
\end{equation}
where the first upper bound on the right-hand side of \eqref{May31_11:57pm} is only available in the range $2.6 \leq A < \infty$.
\end{itemize}

\end{itemize}
\end{theorem}

\noindent {\it Remark:} The function 
\begin{equation}\label{April27_11:27am}
H(x) = \frac{\cos 2\pi x}{1 - 16x^2}
\end{equation}
belongs to the class $\mc{E}$ and verifies $\|H\|_1 = 1/c_0$. We then have $H(0)/\|H\|_1 = c_0 = 1.07995\ldots$, and this yields a slightly inferior lower bound for ${\mc C}(\infty)$ when compared to the one in \eqref{April14_10:41am} (which is obtained in \cite{Gor2} by means of more complicated numerical examples). Due to its simplicity, this particular function $H(x)$ plays an important role in our work, being used in the proof of the lower bound in \eqref{May31_11:57pm} and in the proof of Theorem \ref{Thm5}.

\begin{theorem} \label{Thm2}
Let $1 \leq A \leq \infty$. With respect to problems \eqref{Extremal_Problem_2} and \eqref{Extremal_Problem_2_infty}, the following propositions hold:
\begin{itemize}
\item[(a)] If $A = \infty$, then:

\begin{itemize}
\item[(a.1)] There exists a function $G \in {\mc E}^+$ that extremizes \eqref{Extremal_Problem_2_infty}. 
\item[(a.2)] The sharp constant ${\mc C}^+(\infty)$ verifies the inequality
\begin{equation*}
{\mc C}(\infty) \, \leq  \,{\mc C}^+(\infty) \, < \, 1.2.
\end{equation*}
\end{itemize}

\item[(b)] If $A=1$ then ${\mc C}^+(1)=2$, but there are no extremizers for \eqref{Extremal_Problem_2}.

\item[(c)] If $1 < A < \infty$, then:
\begin{itemize}
\item[(c.1)] There exists an even and real-valued function $G \in {\mc A}^+$ that extremizes \eqref{Extremal_Problem_2}. 
\item[(c.2)] The sharp constant ${\mc C}^+(A)$ verifies the inequality
\begin{equation}\label{June01_3:17pm}
{\mc C}(A) \leq \, \,{\mc C}^+(A)\, \, \leq \, \min\left\{ \left(\frac{1.2}{1 - \frac{0.222}{(A-1)}}\right),\,2\right\},
\end{equation}
where the first upper bound on the right-hand side of \eqref{June01_3:17pm} is only available in the range $1.222 < A < \infty$. 
\item[(c.3)] In particular, if $A=36/11$ a numerical example yields the lower bound
\begin{equation}\label{May31_12:59pm}
\frac{25}{21} < {\mc C}^+\!\left(\frac{36}{11}\right). 
\end{equation}

\end{itemize}

\end{itemize}

\end{theorem}

\noindent{\it Remark:} Note that for small values of $A$, the right-hand side of \eqref{June01_3:17pm} gives a better bound than the right-hand side of \eqref{May31_11:57pm}, and can be used instead. The reason, as we shall see, is that such bounds come from modifying the test functions in the dual problem for the case $A=\infty$. In our construction, these modifications do not necessarily maintain the hierarchy as $A$ approaches $1$.  
 
\subsection{Bounds for prime gaps on RH} Let $p_n$ denote the $n$th prime. Assuming the Riemann hypothesis (RH), a classical result of Cram\'{e}r \cite{C1} yields the bound 
\begin{equation}\label{April17_1:23pm}
\limsup_{n \to \infty} \frac{p_{n+1} - p_{n}}{\sqrt{p_n} \,\log p_n}\leq c\,,
\end{equation}
where $c$ is a universal constant. Building upon the works of Goldston \cite{G} and of Ramar\'{e} and Saouter \cite{RS}, the current best form of this bound is due to Dudek \cite[Theorem 1.3]{D}, who obtained \eqref{April17_1:23pm} with constant $c=1$. Here we improve this and other bounds in this theory by establishing an interesting connection with the extremal problems presented in the previous section.

\smallskip

Our strategy consists of three main ingredients: (i) the explicit formula, (ii) the Brun-Titchmarsh inequality, and (iii) the derived extremal problems in Fourier analysis. Letting $\pi(x)$ denote the number of primes less than or equal to $x$, we define the Brun-Titchmarsh constant ${\bf B}$ in our desired scale by
\begin{equation}\label{May31_2:05pm}
{\bf B} := \limsup_{x \to \infty} \frac{\pi(x + \sqrt{x}) - \pi(x)}{\sqrt{x}/\log x}
\end{equation}
and we observe that
\begin{equation}\label{May31_12:58pm}
1 \leq {\bf B} \leq \frac{36}{11}.
\end{equation}
The lower bound in \eqref{May31_12:58pm} follows from the prime number theorem $\pi(x) \sim x/\log x$ as $x\to\infty$ and the upper bound on {\bf B} follows from the work of Iwaniec \cite[Theorem 14]{I}.

\smallskip

We prove the following general result.
\begin{theorem}\label{Thm3}
Assume the Riemann hypothesis. Let ${\mc C}^+(\cdot)$ be defined in \eqref{Extremal_Problem_2} and ${\bf B}$ be defined in \eqref{May31_2:05pm}. Then, for any $\alpha \geq 0$, we have
\begin{equation}\label{Main_ineq}
\inf\left\{c>0; \ \liminf_{x \to \infty} \frac{\pi\big(x + c \sqrt{x}\log x\big ) - \pi(x)}{\sqrt{x}} > \alpha\right\} \leq \frac{(1 + 2\alpha)} {{\mc C}^+({\bf B})} < \frac{21}{25}(1+2\alpha).
\end{equation}
The last inequality comes from \eqref{May31_12:59pm} and \eqref{May31_12:58pm}.
\end{theorem}
The case $\alpha =0$ in Theorem \ref{Thm3}  yields an affirmative answer for a question posed in \cite{D}, on whether one could  establish \eqref{April17_1:23pm} with a constant $c<1$.

\begin{corollary}\label{CramerCor}
Assume the Riemann hypothesis. Let ${\mc C}^+(\cdot)$ be defined in \eqref{Extremal_Problem_2} and ${\bf B}$ be defined in \eqref{May31_2:05pm}. Then
\begin{equation}\label{June22_6:44am}
\limsup_{n \to \infty} \frac{p_{n+1} - p_{n}}{\sqrt{p_n} \,\log p_n}\leq  \frac{1} {{\mc C}^+({\bf B})} < \frac{21}{25}.
\end{equation}
\end{corollary}

\smallskip

\noindent We note from \eqref{May31_12:58pm} and Theorem \ref{Thm2} (b) that the limit of this method would yield a constant $\frac{1}{2}$ on the right-hand side of \eqref{June22_6:44am}. On the other hand, under stronger assumptions, namely the Riemann hypothesis and Montgomery's pair correlation conjecture, it is known that the limit supremum in \eqref{June22_6:44am} is actually zero (see, for instance, \cite{HB, HBG, Mu}). 

\smallskip

 The case $\alpha =1$ in Theorem \ref{Thm3} yields the constant 
$$c = \frac{3} {{\mc C}^+({\bf B})} < \frac{63}{25}$$ 
on the right-hand side of \eqref{Main_ineq}. This also sharpens the previous best result,  due to Dudek \cite{D}, who had obtained this inequality with constant $c=3$.

\smallskip
 
By working with a particular dilation of the bandlimited function \eqref{April27_11:27am} and an explicit version of the Brun-Titchmarsh inequality due to Montgomery and Vaughan \cite{MV}, we are able to make all of our error terms effective and, assuming the Riemann hypothesis, prove that
\[
p_{n+1}-p_n \le  \frac{22}{25}\sqrt{p_n}\log p_n
\]
for all primes $p_n>3$. 

\begin{theorem}\label{Thm5}
Assume the Riemann hypothesis. Then, for $x\ge 4$, there is always a prime number in the interval $[x,\, x + \frac{22}{25}\sqrt{x}\log x]$. 
\end{theorem}
This theorem improves a result of Dudek, Greni\'{e}, and Molteni \cite[Theorem 1.1]{DGM}, who had previously reached a similar conclusion with $c=\frac{22}{25}$ replaced by $c=c(x) = 1 + \frac{4}{\log x}$. Cram\'{e}r \cite{C2} has conjectured that 
$$p_{n+1} - p_n = O(\log^2 p_n),$$
and this problem remains open to this date. It has been verified by Oliveira e Silva, Herzog, and Pardi \cite[Section 2.2]{OHP} that 
\begin{align}\label{April27_12:11pm}
p_{n+1} - p_n < \log^2 p_n
\end{align}
for all primes $11 \leq p_n \leq 4\cdot 10^{18}$. Estimate \eqref{April27_12:11pm} plainly implies the conclusion of Theorem \ref{Thm5} for all $4\leq x \leq 4\cdot 10^{18}$. Therefore, in our proof, we assume that $x \geq 4\cdot 10^{18}$. 

\smallskip

We now proceed to the proofs of the main results stated in this introduction. This is carried out in Sections \ref{Jun15_Existence} -- \ref{Sec5}. In Section \ref{Sec7} we have a general discussion on some related extremal problems in Fourier analysis, which includes for example the existence of extremizers for the Fourier optimization problem of Cohn and Elkies \cite{CE} related to sphere packing. Some of this material may be of independent interest.

\section{Existence of extremizers}\label{Jun15_Existence}

In this section we discuss the existence of extremizers for the extremal problems \eqref{Extremal_Problem_1} -- \eqref{Extremal_Problem_2_infty}. We prove here parts (a.1), (b), and (c.1) of Theorems \ref{Thm1} and \ref{Thm2}. We begin by making some simplifying observations, that will be helpful for the rest of the paper. Note that we may restrict ourselves to the situation when $\widehat{F} \in L^1(\R)$ (otherwise the quotients on right-hand sides of \eqref{Extremal_Problem_1}, \eqref{Extremal_Problem_2}, and \eqref{Extremal_Problem_2_infty} yield $-\infty$), and we assume this throughout the rest of the paper. In particular, $F$ decays at infinity and $\|F\|_{\infty}$ is attained at some point. 

\smallskip

The class ${\A}$ in  Theorem \ref{Thm1} includes complex-valued functions, but for our extremal problems we can 
restrict attention to even, real-valued functions.  Indeed, given a non-identically zero $F \in \mc{A}$, 
 the following steps either increase the quotients on the right-hand sides of  \eqref{Extremal_Problem_1} -- \eqref{Extremal_Problem_1_infty} or leave them unaltered: 
\begin{itemize}
\item by translating $F$ over $\R$, we may assume that $|F(0)| = \|F\|_{\infty}$;
\item by dilating $F$, we may assume that $\|F\|_1 =1$;
\item by multiplying $F$ by a unimodular complex number, we may assume that $F(0) >0$;
\item by replacing $F(x)$ by $\big(F(x) + \overline{F(x)}\big)/2$ we may assume that $F$ is real-valued;
\item by replacing $F(x)$ by $\big(F(x) + F(-x)\big)/2$ we may assume that $F$ is even.
\end{itemize}

From the definitions it is clear that $\mc{C}(A)$ and $\mc{C}^+(A)$ are non-increasing functions of $A$.  The observations 
above show that in \eqref{Extremal_Problem_1} -- \eqref{Extremal_Problem_1_infty} we can restrict attention to even, real-valued functions, so that 
${\mc{C}}(A) \le \mc{C}^+(A)$.    The Fej\'{e}r kernel $F(x) =  \big(\sin(\pi x)/(\pi x) \big)^2$ reveals that $\mc{C}(\infty) \ge 1$.  
For every $F \in \mc{A}$ we have
\begin{align}\label{June15_10:12am}
|F(0)| - \int_{[-1,1]^c} \big|\widehat{F}(t)\big|\,\dt \leq  \left|\int_{-1}^{1} \widehat{F}(t)\right| \,\dt \leq 2\|F\|_1, 
\end{align} 
so that  $\mc{C}(1) \leq2$.   A similar argument gives $\mc{C}^+(1) \le 2$.   
Putting together all of these observations,  for $1\le A\le \infty$, we obtain the chain of inequalities
\begin{equation*} 
1\le \mc{C}(\infty) \le \mc{C}(A) \le \mc{C}^+(A) \le \mc{C}^+(1) \le 2. 
\end{equation*} 

\subsection{Proof of Theorem \ref{Thm1} (a.1)} This is the case $A= \infty$ and we are restricted to the class $\E \subset \mathcal{A}$ of continuous functions $F:\R \to \C$, with $F \in L^1(\R)$ and $\supp\big( \widehat{F}\big) \subset [-1,1]$. Let $\{F_n\}_{n\geq1}$ be an extremizing sequence verifying the conditions above, i.e. a sequence $\{F_n\}_{n\geq1} \subset \E$ of even and real-valued functions, with $\|F_n\|_1 = 1$,  $\|F_n\|_{\infty} = F_n(0) >0$, and 
\begin{equation*}
\lim_{n \to \infty} F_n(0) = \mc{C}(\infty).
\end{equation*} 
Since ${\mc{C}}(\infty) \le 2$, it follows that $\{F_n\}_{n\geq1}$ is a bounded sequence in $L^2(\R)$.  Hence, there exists $G \in L^2(\R)$ such that (after passing to a subsequence, if necessary) $F_n \rightharpoonup G$ weakly in $L^2(\R)$. In this case, $\supp\big( \widehat{G}\big) \subset [-1,1]$ and by Fourier inversion $G$ can be taken continuous. For any $y \in \R$, we have
\begin{align*}
F_n(y)  = \int_{-1}^{1} e^{2\pi i y t}\, \widehat{F}_n(t)\,\dt & = \int_{-\infty}^{\infty}  \frac{\sin 2 \pi (x - y)}{\pi (x - y)} \,F_n(x)\,\dx  \to \int_{-\infty}^{\infty}  \frac{\sin 2 \pi (x - y)}{\pi (x - y)} \,G(x)\,\dx\\
& = \int_{-1}^{1} e^{2\pi i y t}\, \widehat{G}(t)\,\dt = G(y),
\end{align*}
as $n \to \infty$. It follows that $G$ is even, real-valued and $G(0) = \mc{C}(\infty)$. Moreover, by Fatou's lemma, we have $\|G\|_1 \leq 1$. Hence $G \in \mc{E}$, and from the definition of $\mc{C}(\infty)$ we must have $\|G\|_1 = 1$ which makes $G$ an extremizer. Multiplying this $G$ by the constant factor $\mc{C}(\infty)^{-1}$ we arrive at the extremizer stated in the theorem (that assumes the value $1$ at $x=0$).

\subsection{Proof of Theorem \ref{Thm1} (b)} \label{June08_10:45am}

We already observed in \eqref{June15_10:12am} that  $\mc{C}(1) \leq2$. 
By taking $F_{\varepsilon}(x) = \frac{1}{\sqrt{\varepsilon}  }\,e^{-\pi x^2/\varepsilon}$ with $\varepsilon \to 0^+$ we see that $\mc{C}(1) =2$. In order to obtain equality in \eqref{June15_10:12am} we must have $\widehat{F}(t) = c\|F\|_1$ for all $t \in [-1,1]$, for some constant $c \in \C$ with $|c|=1$. This is not possible, and hence there are no extremizers in this case.

\subsection{Proof of Theorem \ref{Thm1} (c.1)} \label{June08_11:01am}
Here $1 < A < \infty$. Suppose $F \in \mc{A}$ is non-identically zero, with 
\begin{equation} \label{2.3} 
\frac 12 \le \frac{|F(0)|  - A \int_{[-1,1]^c} \big|\widehat{F}(t)\big|\,\dt}{\|F\|_1}. 
\end{equation} 
Since 
\begin{align*}
\int_{[-1,1]^c} \big|\widehat{F}(t)\big|\,\dt \geq \left| F(0) - \int_{-1}^{1} \widehat{F}(t)\,\dt\right| \geq |F(0)| - 2\|F\|_1\,,
\end{align*}
we may use \eqref{2.3} to see that 
\begin{equation} \label{2.4}
|F(0)| \le \frac{2A-\tfrac{1}{2}}{A-1} \| F\|_1. 
\end{equation} 
Inserting this estimate into \eqref{2.3}, we also have 
\begin{equation} \label{2.5} 
\int_{[-1,1]^{c}} |{\widehat F}(t)| \dt \le \frac{3}{2(A-1)}\| F\|_1. 
\end{equation}  


\smallskip

Let $\{F_n\}_{n\geq1} \subset \mc{A}$ be an extremizing sequence of even and real-valued functions, with $\|F_n\|_1 = 1$,  $\|F_n\|_{\infty} = F_n(0) >0$, and $\widehat{F}_n \in L^1(\R)$. Thus 
\begin{equation*}
\lim_{n \to \infty} \left(F_n(0)  - A \int_{[-1,1]^c} \big|\widehat{F}_n(t)\big|\,\dt \right)= \mc{C}(A).
\end{equation*} 
Since $\mc{C}(A) \ge 1$, from our observation in \eqref{2.4}  we see that 
$\{ F_n(0)\}_{n\ge 1}$ is a bounded sequence, and from \eqref{2.5} that $\big\{\big\|\widehat{F}_n\big\|_1\big\}_{n\geq1}$ is also bounded. 




\subsubsection{Step 1} Since $\|F_n\|_{\infty} = F_n(0)$, the sequence $\{F_n\}_{n \geq1}$ is bounded in $L^2(\R)$. Passing to a subsequence, if necessary, we may assume that $F_n(0) \to c$, for some constant $c \geq {\mc C}(A)$, and that $F_n \rightharpoonup G$ weakly in $L^2(\R)$ for some $G \in L^2(\R)$. By Mazur's lemma \cite[Corollary 3.8 and Exercise 3.4]{B}, there exists a sequence $H_k \to G$ strongly in $L^2(\R)$, with $H_k \in {\rm Conv}\big(\{F_n\}_{n \geq k}\}\big)$ (i.e. each $H_k$ is a {\it finite} linear convex combination of functions $F_n$ with $n \geq k$). Note that $H_k$ is even and real-valued, $\|H_k\|_{\infty} = H_k(0) \to c$, $\|H_k\|_1 \leq 1$, and $\big\{\big\|\widehat{H}_k\big\|_1\big\}_{k\geq1}$ remains bounded. By passing to a further subsequence, we may also assume that $H_k \to G$ and $\widehat{H}_k \to \widehat{G}$, pointwise almost everywhere. Hence $G$ is also even and real-valued. Note that $\{H_k\}_{k\geq1}$ is also an extremizing sequence.

\subsubsection{Step 2} By Fatou's lemma $\|G\|_1 \leq \liminf_{k \to \infty}\|H_k\|_1 \leq 1$ and $\|\widehat{G}\|_1 \leq \liminf_{k \to \infty}\|\widehat{H}_k\|_1 <\infty$. By Fourier inversion, we may assume that $G$ is continuous (after eventually modifying it on a set of measure zero), hence $G \in \mc{A}$. First we claim that $G$ is nonzero. In fact, since $\{H_k\}_{k\geq1}$ is an extremizing sequence and $H_k(0) \to c \geq {\mc C}(A)$, from \eqref{2.4} we find that $\liminf_{k \to \infty} \|H_k\|_1 \geq c_1 >0$. From the $L^2$-convergence (applied below just in the interval $[-1,1]$) and Fatou's lemma, we have
\begin{align*}
G(0)  & - A \int_{[-1,1]^c}  \big|\widehat{G}(t)\big|\,\dt =\int_{-1}^{1}\widehat{G}(t) \,\dt - \int_{[-1,1]^c} \left(\big|\widehat{G}(t)\big| - \widehat{G}(t)\right) \dt - (A-1) \int_{[-1,1]^c} \big|\widehat{G}(t)\big|\,\dt\\
& \geq \limsup_{k\to \infty} \left(\int_{-1}^{1}\widehat{H_k}(t) \,\dt - \int_{[-1,1]^c} \left(\big|\widehat{H}_k(t)\big| - \widehat{H}_k(t)\right)\dt - (A-1) \int_{[-1,1]^c} \big|\widehat{H}_k(t)\big|\,\dt\right)\\
& = \limsup_{k\to \infty} \left( H_k(0)  - A \int_{[-1,1]^c}  \big|\widehat{H}_k(t)\big|\,\dt\right)\\
& \geq  c_1 \, \mc{C}(A).
\end{align*}
This shows that $G$ is nonzero. The same computation above (up to its third line) shows that $G$ is indeed an extremizer, since $\|G\|_1 \leq \liminf_{k \to \infty}\|H_k\|_1$.

\subsection{Proof of Theorem \ref{Thm2} (a.1), (b), and (c.1)} \label{June16_8:25am} The proof of part (b) follows along the same lines as the argument in \S\ref{June08_10:45am} (with the same extremizing family). The proofs of parts (a.1) and (c.1) follow the outline of \S \ref{June08_11:01am} and we simply indicate the minor modifications needed.

\smallskip

In seeking extremizers when $1 < A < \infty$, we may assume that $F(0) >0$ and that $\widehat{F} \in L^1(\R)$ (recall that here we are already working within the class of even and real-valued functions).  Suppose that $F \in \mc{A}^+$ is non-identically zero, with 
\begin{equation} \label{2.3_thm2} 
\frac 12 \le \frac{F(0)  - A \int_{[-1,1]^c} \big(\widehat{F}(t)\big)_+\,\dt}{\|F\|_1}. 
\end{equation} 
Since 
\begin{align*}
\int_{[-1,1]^c} \big(\widehat{F}(t)\big)_+\,\dt \geq F(0) - \int_{-1}^{1} \widehat{F}(t)\,\dt \geq F(0) - 2\|F\|_1\,,
\end{align*}
we may use \eqref{2.3_thm2} to see that 
\begin{equation} \label{2.4_thm2}
F(0) \le \frac{2A-\tfrac{1}{2}}{A-1} \| F\|_1. 
\end{equation} 
As before, inserting this estimate into \eqref{2.3_thm2}  we obtain
\begin{equation} \label{2.5_thm2} 
\int_{[-1,1]^{c}} \big({\widehat F}(t)\big)_+ \dt \le \frac{3}{2(A-1)}\| F\|_1. 
\end{equation}  

\smallskip

Let $\{F_n\}_{n\geq1} \subset \mc{A^+}$ be an extremizing sequence with $\|F_n\|_1 = 1$,  $F_n(0) >0$, and $\widehat{F}_n \in L^1(\R)$. Note that, in principle, we do not necessarily have $ \|F_n\|_{\infty} = F_n(0)$. Since $\mc{C}^+(A) \ge 1$, from \eqref{2.4_thm2}  we see that 
$\{ F_n(0)\}_{n\ge 1}$ is a bounded sequence, and from \eqref{2.5_thm2} we see that $\big\{\big\|\widehat{F}_n\big\|_1\big\}_{n\geq1}$ is also bounded. 

\smallskip

The rest of the proof follows as in Steps 1 and 2 of \S \ref{June08_11:01am}. Note that the corresponding sequence $\{H_k\}_{k\geq1}$ will be extremizing,  due to the general property that $(f+g)_+ \leq f_+ + g_+$, and inequality \eqref{2.4_thm2} shows that $\liminf_{k \to \infty} \|H_k\|_1 \geq c_1 >0$. For the final computation, one uses the identity
\begin{align*}
G(0)  - A \int_{[-1,1]^c}  \big(\widehat{G}(t)\big)_+\,\dt =\int_{-1}^{1}\widehat{G}(t) \,\dt - \int_{[-1,1]^c}  - \big(\widehat{G}(t)\big)_- \dt - (A-1) \int_{[-1,1]^c} \big(\widehat{G}(t)\big)_+\,\dt.
\end{align*}

\smallskip

For the case $A = \infty$ (part (a.1)), the required modifications are similar and we omit the details.

\section{Uniqueness of extremizers}\label{Sec3_Uniqueness}

In this section we continue the study of the extremal problem \eqref{Extremal_Problem_1_infty}. We prove the uniqueness of a bandlimited extremizer (up to multiplication by a complex scalar) and provide its variational characterization as described in parts (a.2) and (a.3) of Theorem \ref{Thm1}.

\subsection{Proof of Theorem \ref{Thm1} (a.2)} Let $G \in \E \subset {\mc A}$ be an even and real-valued extremizer of \eqref{Extremal_Problem_1_infty} with $G(0)=1$. Let $G_1 \in \E$ be another extremizer of \eqref{Extremal_Problem_1_infty}, with $G_1(0) = 1$. It suffices to show that $G_1 = G$.  

\smallskip

Let $F = (G + G_1)/2$. Then, by the triangle inequality, we have
\begin{equation}\label{April19_10:55am}
\int_{-\infty}^{\infty} |F(x)|\,\dx \leq \frac{1}{2} \int_{-\infty}^{\infty} \big(|G(x)| + |G_1(x)|\big)\,\dx= \frac{1}{{\mc C}(\infty)},
\end{equation}
and $F(0) = 1$. To avoid strict inequality in \eqref{April19_10:55am} we must have
$$|G(x) + G_1(x)| = |G(x)| + |G_1(x)|$$
for all $x \in \R$. In particular, this shows that $G_1:\R \to \C$ is real-valued and that 
$$G(x)\,G_1(x) \geq 0$$ 
for all $x \in \R$. Let $R= G\cdot G_1$. Then $R$ is a nonnegative and integrable function with $\supp\big(\widehat{R}\big) \subset [-2,2]$. By a classical result of Krein \cite[p. 154]{A}, we have $R(x) = |S(x)|^2$, for some $S \in L^2(\R)$ with $\supp\big(\widehat{S}\big) \subset [-1,1]$. Observe that $|S(0)| =1$ and that
\begin{equation}\label{April19_11:39am}
\int_{-\infty}^{\infty} |S(x)|\,\dx = \int_{-\infty}^{\infty} \sqrt{G(x)\,G_1(x)}\,\dx \leq \frac{1}{2} \int_{-\infty}^{\infty} \big(|G(x)| + |G_1(x)|\big)\,\dx  = \frac{1}{{\mc C}(\infty)}.
\end{equation}
In particular $S \in L^1(\R)$. To avoid strict inequality in \eqref{April19_11:39am} we must have $G(x) = G_1(x)$ for all $x \in \R$, completing the proof.
 
\subsection{Proof of Theorem \ref{Thm1} (a.3)} Let $G$ be the unique extremal function of \eqref{Extremal_Problem_1_infty} with $G(0)=1$. Let $F \in \E$ be a real-valued function with $F(0) = 0$ and define, for $\varepsilon \in \R$,
$$\Phi(\varepsilon) := \int_{-\infty}^{\infty} |G(x) + \varepsilon F(x)|\,\dx = \int_{-\infty}^{\infty} \big((G(x) + \varepsilon F(x))^2\big)^{1/2}\,\dx.$$
This is a differentiable function of the variable $\varepsilon$ and, since $G$ is an extremizer, we must have
\begin{align}\label{April19_12:11pm}
0 = \frac{\partial \Phi}{\partial \varepsilon}(0) = \int_{-\infty}^{\infty} \sgn(G(x))\,F(x)\,\dx.
\end{align}
If $F_1 \in \E$ is a generic real-valued function (not necessarily with $F_1(0) =0$), by \eqref{April19_12:11pm} we obtain that
\begin{align} \label{April19_12:17pm}
\begin{split}
{\mc C}(\infty) \int_{-\infty}^{\infty} \sgn(G(x))\,F_1(x)\,\dx & = {\mc C}(\infty) \int_{-\infty}^{\infty}\sgn(G(x)) \big(F_1(x) - F_1(0)G(x)\big)\dx  \\
& \ \ \ \ \ \ \ \ \ \ \ \ \ \ \ \ \ \  + {\mc C}(\infty) \int_{-\infty}^{\infty} \sgn(G(x))\,F_1(0)\,G(x)\,\dx\\
& = F_1(0).
\end{split}
\end{align}
Finally, if $F_2 \in \E$ is a generic complex-valued function, we may write $F_2(x) = A(x) - iB(x)$, where $A(x) = (F_2(x) + \overline{F_2(x)})/2$ and $B(x)  = i(F_2(x) - \overline{F_2(x)})/2$ are real-valued functions in $\E$, and use \eqref{April19_12:17pm} to arrive at
\begin{equation*}
{\mc C}(\infty) \int_{-\infty}^{\infty} \sgn(G(x))\,F_2(x)\,\dx =   F_2(0).
\end{equation*}

\section{Upper and lower bounds}

In this section we conclude the proofs of Theorems \ref{Thm1} and \ref{Thm2} by establishing the proposed upper and lower bounds for the sharp constants ${\mc C}(A)$ and ${\mc C}^+(A)$. 

\subsection{Approximations} \label{Sec_Approx} For the purpose of finding the values of the sharp constants ${\mc C}(A)$ and ${\mc C}^+(A)$ in problems \eqref{Extremal_Problem_1} -- \eqref{Extremal_Problem_2_infty}, without loss of generality we may work with smooth functions. For instance, let us show that we can simply consider $\widehat{F} \in C_c^{\infty}(\R)$. This observation is useful in some passages later in the paper.

\smallskip

Starting with $0 \neq F \in {\mc A}$ (or $0 \neq F \in {\mc A}^+$ in the case of \eqref{Extremal_Problem_2}), we write 
\begin{equation*}
J(F) :=  \frac{|F(0)| - A\int_{[-1,1]^c} \big|\widehat{F}(t)\big|\,\dt}{\|F\|_{1}}  \quad \text{and} \quad J^+(F) :=  \frac{F(0) - A\int_{[-1,1]^c} \big(\widehat{F}(t)\big)_+\,\dt}{\|F\|_{1}}.
\end{equation*}
In either situation we may also assume that $\widehat{F} \in L^1(\R)$ and that $J(F)$ and $J^+(F)$ are positive. Let $K(x) = \big(\sin(\pi x)/(\pi x) \big)^2$ be the Fej\'{e}r kernel and, for $\lambda >0$, define $K_\lambda(x) = \lambda^{-1}K(x/\lambda)$. By Young's inequality we have $\|F*K_{\lambda}\|_1 \leq \|F\|_1$, and using dominated convergence it follows that $\limsup_{\lambda \to 0} J(F*K_{\lambda}) \geq J(F)$ and $\limsup_{\lambda \to 0} J^+(F*K_{\lambda}) \geq J^+(F)$. Hence we may assume that our test function $F$ is bandlimited.

\smallskip

Let $\eta \in C^{\infty}_c(\R)$ be an even, nonnegative, and radially non-increasing function such that $\eta(0) =1$, $\supp(\eta) \subset [-1,1]$, and $\int_{-1}^1 \eta(x)\,\dx =1$. Again, let $\eta_\lambda(x) = \lambda^{-1}\eta(x/\lambda)$. If $\supp(\widehat{F}) \subset[-\Lambda, \Lambda]$, then $\widehat{F}*\eta_{\lambda} \in C^{\infty}_c(\R)$ and $\supp(\widehat{F}*\eta_{\lambda}) \subset [-\Lambda - \lambda, \Lambda + \lambda]$. By dominated convergence, we have $\lim_{\lambda \to 0} J(F\cdot \widehat{\eta_{\lambda}})= J(F)$ and $\lim_{\lambda \to 0} J^+(F\cdot \widehat{\eta_{\lambda}})= J^+(F)$. This verifies our claim in the cases $1 \leq A < \infty$. In the cases $A=\infty$ one has to slightly dilate $F$ beforehand in order to apply the procedure above and arrive at a function in the class $\E \subset {\mc A}$ for \eqref{Extremal_Problem_1_infty} and $\E^+  \subset {\mc A}^+$ for \eqref{Extremal_Problem_2_infty}.

\subsection{Proof of Theorem \ref{Thm1} (a.4)} \label{Sec4.2} The bounds
\begin{equation}\label{August10_3:37pm}
1.08185\ldots \leq  {\mc C}(\infty) \leq  1.09769\ldots
\end{equation}
were proved in the very interesting work of Gorbachev \cite[Theorem 3]{Gor2}, to which we refer the reader for details. These bounds improved upon the work of Andreev, Konyagin, and Popov \cite{AKP}, who had previously obtained
\begin{equation}\label{August10_3:38pm}
c_0 = 1.07995\ldots \leq {\mc C}(\infty) \leq 1.17898.
\end{equation}

As already pointed out in the introduction, the lower bound in \eqref{August10_3:38pm} comes from the simple example
$$ 
H(x) = \frac{\cos (2\pi x)}{1-16x^2}.
$$  
The Fourier transform of $H$ is $\widehat{H}(t) = \frac {\pi}{4} \cos(\pi t/2) \chi_{[-1,1]}(t)$, which 
may be verified by starting with our expression for $\widehat{H}(t)$ and computing its Fourier transform to recover $H$.   Thus $H$ belongs 
to the class $\E$, and $H(0)$ is clearly $1$.  To compute the $L^1$-norm of $H$ we observe that $\sgn(H(x)) = 2\chi_{[-\frac14, \frac14]}(x) - \sgn(\cos 2\pi x)$, and use Plancherel's theorem and the fact that $\sgn(\cos 2\pi x)$ has distributional Fourier transform supported outside $(-1,1)$ to get\footnote{The function $x \mapsto \sgn(\cos 2\pi x)$ is an example of a  {\it high pass function}, as studied in \cite{Lo}.} 
\begin{align*}
\| H\|_1 & = \int_{-\infty}^{\infty} |H(x)|\, \dx = \int_{-\infty}^{\infty} \big( 2\chi_{[-\frac14, \frac14]}(x) - \sgn(\cos 2\pi x)\big) \,H(x)\,\dx = \int_{-\infty}^{\infty}  2\chi_{[-\frac14, \frac14]}(x) \,H(x)\,\dx\\
& = \int_{-1}^{1} \left( \frac{2\sin(\pi t /2)}{\pi t} \right) \left( \frac {\pi}{4} \cos(\pi t/2)\right) \dt = \frac{\pi}{4} \int_{-1}^{1} \frac{\sin \pi t}{\pi t} \, \dt =1/c_0.
\end{align*}
This example will be useful later on to generate lower bounds for ${\mc C}(A)$ in the general case $1 < A < \infty$. The lower bound of Gorbachev \cite{Gor2} in \eqref{August10_3:37pm} comes from more complicated numerical examples.

\smallskip

The upper bound in \eqref{August10_3:37pm} comes from a dual formulation of the problem. Suppose that $\psi \in L^{\infty}(\R)$ is such that its distributional Fourier transform is identically equal to $1$ on the interval $(-1,1)$. Let ${\mc S}(\R)$ denote the Schwartz class. Then, for $F \in \E \cap {\mc S}(\R)$ (as discussed in \S \ref{Sec_Approx}), we have
\begin{align*}
\|\psi\|_{\infty} \int_{-\infty}^{\infty}|F(x)| \,\dx \geq \left|\int_{-\infty}^{\infty} F(x)\,\psi(x) \,\dx \right| = \left|\int_{-1}^{1} \widehat{F}(t)\,\widehat{\psi}(t) \,\dt \right| = |F(0)|,
\end{align*}
which implies that
\begin{equation*}
 {\mc C}(\infty) \leq \|\psi\|_{\infty}.
 \end{equation*}
With this dual formulation, it suffices to exhibit a nice test function $\psi$. 

\smallskip

We now briefly describe the construction of Gorbachev \cite[Lemma 9]{Gor2}. To simplify the notation (and align with the terminology of \cite{Gor2} to facilitate the references) we let 
$$j(x) = \frac{\sin(2 \pi x)}{2\pi x}$$ 
in what follows. For $\tau = 29289/100000 = 0.29289$ we define a continuous and piecewise linear function $\alpha:[0,1/2] \to \R$ by 
\begin{align}\label{August11_11:54am}
\begin{split}
\alpha(x) = \left\{
\begin{array}{lcl}
2x-1, && 0 \leq x \leq \tau;\\
2\tau - 1 + 2(1-\tau)(x - \tau)/\varepsilon, & &\tau \leq x \leq \tau + \varepsilon;\\
1, & &\tau + \varepsilon \leq x \leq 1/2 - 2 \varepsilon;\\
1 - y(x - 1/2 + 2\varepsilon)/\varepsilon, && 1/2 - 2\varepsilon \leq x \leq 1/2 - \varepsilon;\\
1 - y + y(x - 1/2 + \varepsilon)/\varepsilon, && 1/2 - \varepsilon \leq x \leq 1/2,
\end{array}
\right.
\end{split}
\end{align}
where 
\begin{align}\label{August11_11:55am}
\varepsilon = \frac{\tau^2 - 2 \tau + 1/2}{1 + y - 2\tau} >0,
\end{align}
and, having defined \eqref{August11_11:54am} and \eqref{August11_11:55am}, $y$ is finally chosen so that
$$\int_{0}^{1/2} \frac{ (1 - \alpha(x))}{j(x)} \cos(2 \pi x)\,\dx = 0.$$
One arrives at the values $y = 0.43056\ldots$ and $\varepsilon = 0.0000053884\ldots$. Let 
$$d_0 = \left(\int_0^{1/2} \frac{1 - \alpha(x)}{j(x)}\,\dx \right)^{-1} = 1.09769\ldots$$
and define {\it $1-$periodic even functions} $a(x)$ and $b(x)$ by 
$$a(x) = d_0 \,\alpha(x) \ \ {\rm and} \ \ b(x) = \frac{d_0 - a(x)}{2j(x)} - 1,\ \ {\rm for} \ \ x\in [0,1/2].$$
As observed in \cite{Gor2}, with this construction the functions $a$ and $b$ have Fourier series expansions
$$a(x) = \sum_{n=1}^{\infty}2 a_n \cos(2\pi nx),  \ \ b(x) = \sum_{n=2}^{\infty}2 b_n \cos(2\pi nx),  \ \ \sum_{n=1}^{\infty} |a_n| <\infty,  \ \ \text{and} \ \  \sum_{n=2}^{\infty} |b_n| < \infty.$$
(notice that the first Fourier coefficients verify $a_0 = b_0 = b_1 = 0$). A numerical evaluation leads to
$$a_1 = -0.5622\ldots, \ \ a_2 =0.0684\ldots, \ \ a_3 = 0.1005\ldots, $$
and since $\|a\|_{L^2[-\frac12,\frac12]}^2 = 0.7238\ldots$ and $2a_1^2  = 0.6321\ldots$, an application of Plancherel's theorem gives us that $|a_n| \leq |a_1|$ for all $n$. For the function $b$ we adopt a slightly different approach to bounding the Fourier coefficients $b_n$ (since $\|b\|_{L^2[-\frac12,\frac12]}$ is very large). A numerical integration yields 
$$|b_n| \le \int_{-1/2}^{1/2} |b(x)|\,\dx = 0.8283\ldots $$
for all $n\ge 2$.

\smallskip

Finally, let $\phi(x) = 2j(x)(1 + b(x))$, and define
\begin{equation}\label{June08_4:46pm}
\psi(x) = \phi(x) + a(x).
\end{equation}
This is the test function constructed by Gorbachev \cite{Gor2}, which verifies $\|\psi\|_{\infty} = d_0$ and has distributional Fourier transform identically equal to $1$ on the interval $(-1,1)$. In fact, we have
\begin{align}\label{June08_5:17pm}
\begin{split}
\widehat{\psi}(t) & = \widehat{\phi}(t) + \sum_{n=1}^{\infty} a_n\big(\delta(t-n) + \delta(t+n)\big)\\
& = \chi_{[-1,1]}(t) + \sum_{n=2}^{\infty} b_n \big( \chi_{[-1,1]}(t-n) + \chi_{[-1,1]}(t+n) \big)+ \sum_{n=1}^{\infty} a_n\big(\delta(t-n) + \delta(t+n)\big),
\end{split}
\end{align}
where $\delta$ is the Dirac delta distribution. We shall use this construction to generate upper bounds for ${\mc C}(A)$ in the general case $1 < A < \infty$. The observation that $\|\widehat{\phi}\|_{\infty} =1$ will be relevant later on.

\smallskip

\noindent {\it Remark}: In an earlier version of this manuscript, without being aware of the references \cite{AKP} and \cite{Gor2}, we had initially arrived at the test function
\begin{align*}
\widetilde{\psi}(x) = 2\widetilde{a}_0\chi_{[-\frac14, \frac14]}(x)  +  \sum_{n=1}^{\infty} 2\widetilde{a}_n\left( \chi_{[-\frac{1}{4}, \frac{1}{4}]}(x - \tfrac{n}{2}) + \chi_{[-\frac{1}{4},  \frac{1}{4}]}(x + \tfrac{n}{2})\right) - \widetilde{a_0}\sgn(\cos(2\pi x)),
\end{align*}
where $\widetilde{a}_n =  \frac{4}{\pi} \sum_{j = n}^{\infty} \frac{(-1)^j}{(2j+1)^2}$ are the Fourier coefficients in the expansion
\begin{align*}
\frac{(\pi t /2)}{\sin(\pi t /2)} = \widetilde{a}_0 + 2\sum_{n =1}^{\infty} \widetilde{a}_n \cos(n\pi t)
\end{align*}
for $-1\leq t < 1$. This leads to the bound $\mc{C}(\infty) \leq \big\|\widetilde{\psi}\big\|_{\infty} = \widetilde{a}_0 = 1.16624\ldots$, which is intermediate between \eqref{August10_3:37pm} and \eqref{August10_3:38pm}.

\subsection{Proof of Theorem \ref{Thm1} (c.2)} 

\subsubsection{Lower bounds} \label{Jun14_2:56pm}  As before, let $H(x) = (\cos 2\pi x)/(1 - 16x^2)$. Take $F(x) = H(x/\lambda)$ for a suitable parameter $\lambda \in (0,1]$ to be optimized. Then $F(0) =1$ and $\|F\|_1 = \lambda \|H\|_1 = \lambda /c_0$ with $c_0 = 1.079950...$. The ratio to be maximized is  
\begin{equation*}
\frac{c_0}{\lambda} \left( 1 -  A\,\lambda\,\frac{\pi}{4} \int_{1 \leq |t| \leq \frac{1}{\lambda}} \cos\left(\frac{\pi \lambda t }{2}\right)\,\dt\right) = \frac{c_0}{\lambda} \left( 1 - A\left(1 - \sin\left(\frac{\pi \lambda}{2}\right) \right)\right).
\end{equation*}
Calculus shows that this is maximized by choosing $\lambda$ such that
\begin{equation}\label{June08_4:01pm}
1 - \frac{1}{A} = \sin\left(\frac{\pi \lambda}{2}\right) - \frac{\pi \lambda}{2}\cos\left(\frac{\pi \lambda}{2}\right).
\end{equation}
For $\lambda = \lambda(A)$ verifying \eqref{June08_4:01pm}, this examples demonstrates that
$${\mc C}(A) \geq \frac{\pi Ac_0}{2} \cos\left(\frac{\pi \lambda(A)}{2}\right).$$

Note that as $A \to 1^+$, this lower bound goes to $\pi c_0/2$ and is not very effective. Alternatively, we can then use a dilation of the Fej\'{e}r kernel $K(x) = (\sin(\pi x)/ (\pi x))^2$ (note that $\widehat{K}(t) = (1-|t|)_+$). Again we consider $F(x) = K(x/\lambda)$ and optimize the dilation parameter $\lambda \in (0,1]$. The ratio we seek to maximize is
\begin{equation*}
\frac{1}{\lambda} \left(1 - A\,\lambda\,\int_{1 \leq |t| \leq \frac{1}{\lambda}} (1 - |\lambda t|)_+\, \dt\right) = \frac{1}{\lambda} - A\left(\frac{1}{\lambda} + \lambda - 2\right).
\end{equation*}
The optimal choice is $\lambda = \sqrt{(A-1)/A}$, which leads to the bound
$${\mc C}(A) \geq 2A - 2\sqrt{A(A-1)}.$$

\subsubsection{Upper bounds} \label{Subs_UB} We already know that ${\mc C}(A) \leq {\mc C}(1)=2$. The other upper bound comes from duality considerations. Suppose that $\varphi \in L^{\infty}(\R)$ is such that its distributional Fourier transform is identically equal to $1$ on the interval $(-1,1)$ and $\big|\widehat{\varphi}(t) - 1\big| \leq A$ for all $t \in \R$. Then, for $F \in \mc{A} \cap {\mc S}(\R)$ (as discussed in \S \ref{Sec_Approx}), we have
\begin{align*}
\|\varphi\|_{\infty} \int_{-\infty}^{\infty}|F(x)| \,\dx \geq \left|\int_{-\infty}^{\infty} F(x)\,\varphi(x) \,\dx \right| = \left|\int_{-\infty}^{\infty} \widehat{F}(t)\,\widehat{\varphi}(t) \,\dt \right| \geq |F(0)| - A\int_{[-1,1]^c} \big|\widehat{F}(t)\big|\,\dt.
\end{align*}
This leads to 
$ {\mc C}(A) \leq \|\varphi\|_{\infty}$.

\smallskip

Let $\psi$ be defined by \eqref{June08_4:46pm}. The idea is to mollify this function (used in the case $A = \infty$) in order to ``bring down the delta functions'' in its Fourier transform into the acceptable range $\big|\widehat{\varphi}(t) - 1\big| \leq A$ for all $t \in \R$. First we dilate $\widehat{\psi}$ defined by \eqref{June08_5:17pm}  to push the delta functions away from the interval $[-1,1]$, in other words, for $\gamma >1$, we observe that
\begin{equation*}
\widehat{\psi}(t/\gamma) = \widehat{\phi}(t/\gamma)  + \sum_{n=1}^{\infty} \gamma a_n \big(\delta(t-\gamma n) + \delta(t+\gamma n)\big).
\end{equation*}
Let $R(t) = \chi_{[-1/2,1/2]}(t)$. For $\lambda >0$, we write $R_{\lambda}(t) = \lambda^{-1} R(t/\lambda)$ and define
\begin{align}\label{June09_4:47pm}
\begin{split}
\widehat{\varphi}(t) &:= \big(\widehat{\psi}(\cdot/\gamma) * R_{\lambda}\big)(t) \\
& \,\,= \big(\widehat{\phi}(\cdot/\gamma) * R_{\lambda}\big)(t) + \sum_{n=1}^{\infty} \frac{\gamma a_n}{\lambda}\big(\chi_{[-\frac{\lambda}{2}, \frac{\lambda}{2}]}(t - \gamma n) + \chi_{[-\frac{\lambda}{2}, \frac{\lambda}{2}]}(t + \gamma n)\big).
\end{split}
\end{align}
Recall that $|a_n| \leq |a_1| < 0.6$ for all $n \geq 1$. Let $c = 0.6$, so that all the delta functions in \eqref{June08_5:17pm} have coefficients at most $c$. Let us assume that $A \geq 2 + c$ (so that our particular choices of $\lambda$ and $\gamma$ below verify $0<\lambda \leq \gamma$). We choose $\gamma -1 = \frac{\lambda}{2}$ (so that the support of the mollified delta functions in \eqref{June09_4:47pm} stay away from the interval $(-1,1)$) and $\frac{c\gamma}{\lambda} = A-2$ (so that the height of the mollified delta functions in \eqref{June09_4:47pm} is at most $A-2$). This leads to the explicit forms
$$\lambda = \frac{2}{ \frac{2}{c}(A-2) -1} \ \ \ {\rm and} \ \ \ \gamma =  \frac{1}{1 - \frac{c}{2(A-2)}}.$$
From \eqref{June09_4:47pm} we conclude that $\widehat{\varphi}(t) = 1$ for $t \in (-1,1)$ and, since $\big\|\widehat{\phi}\big\|_{\infty} =1$, we also have $|\widehat{\varphi}(t)| \leq A-1$ for all $t \in \R$, which in particular implies that $\big|\widehat{\varphi}(t) - 1\big| \leq A$ for all $t \in \R$ (note that the mollified delta functions on the right-hand side of \eqref{June09_4:47pm} have disjoint supports due to the fact that $\lambda \leq \gamma$). Since $\varphi(x) = \gamma\, \psi(\gamma x)\,\widehat{R}(x/\lambda)$, our upper bound is then $\|\varphi\|_{\infty} = \gamma \|\psi\|_{\infty} = \gamma \,d_0$.

\subsection{Proof of Theorem \ref{Thm2} (a.2)} We proceed again via duality considerations. Suppose that $\Psi \in L^{\infty}(\R)$ is a real-valued function such that its distributional Fourier transform is identically equal to $1$ on the interval $(-1,1)$ and $ \widehat{\Psi}(t) - 1 \leq 0$ for all $t \in \R$. Then, for $F \in \E^+ \cap {\mc S}(\R)$ (as discussed in \S \ref{Sec_Approx}), we have
\begin{align*}
\|\Psi\|_{\infty} \int_{-\infty}^{\infty}|F(x)| \,\dx \geq \int_{-\infty}^{\infty} F(x)\,\Psi(x) \,\dx  = \int_{-\infty}^{\infty} \widehat{F}(t)\,\widehat{\Psi}(t) \,\dt  \geq F(0),
\end{align*}
which implies that 
\begin{equation*}
 {\mc C}^{+}(\infty) \leq \|\Psi\|_{\infty}.
 \end{equation*}
 Experimentation gave the following numerical example.   Let $a=0.018$, $b=0.027$, and $c=0.002$, and consider 
 \begin{align}\label{June13_11:39am}
 \begin{split}
 \Psi(x) & = \frac{\sin(2 \pi x)}{\pi x} + \frac{2 \sin(a \pi x)}{\pi x} \cos(3 \pi x) + \frac{2\sin(b \pi x)}{\pi x}\cos(4 \pi x) + \frac{2\sin(c \pi x)}{\pi x}\cos(10\pi x)\\
 &  \ \ \ \ \ \ \ \ \ \ - 0.888 \cos(2 \pi x) - 0.01 \cos(6 \pi x), 
 \end{split}
 \end{align}
which has Fourier transform 
  \begin{equation}\label{June13_11:44am}
  \begin{split}
  \widehat{\Psi}(t) & = \chi_{[-1,1]}(t) + \chi_{[-a/2,a/2]}(t-\tfrac32) + \chi_{[-a/2,a/2]}(t+\tfrac32) 
  \\
  &  \qquad  + \chi_{[-b/2,b/2]}(t-2) + \chi_{[-b/2,b/2]}(t+2)  
  \\
  & \qquad + \chi_{[-c/2,c/2]}(t-5) + \chi_{[-c/2,c/2]}(t+5)
  \\
  & \qquad   - 0.444 (\delta(t +1) + \delta(t-1)) - 0.005(\delta(t +3) + \delta(t-3)).
  \end{split}
\end{equation}
For this test function we have
$\|\Psi\|_{\infty} < 1.2.$

\begin{figure} \label{figure2}
\includegraphics[scale=.4]{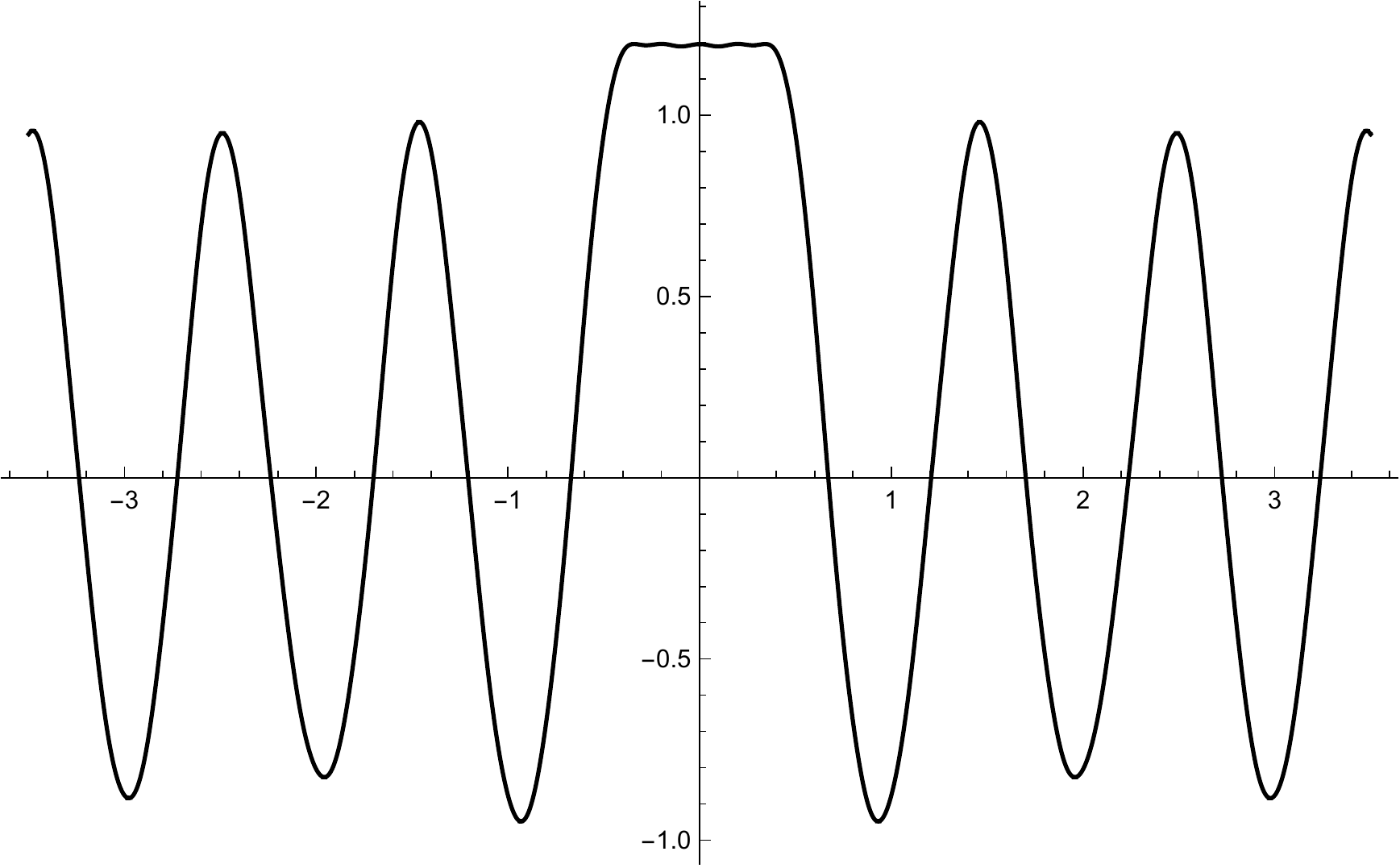} \qquad 
\includegraphics[scale=.4]{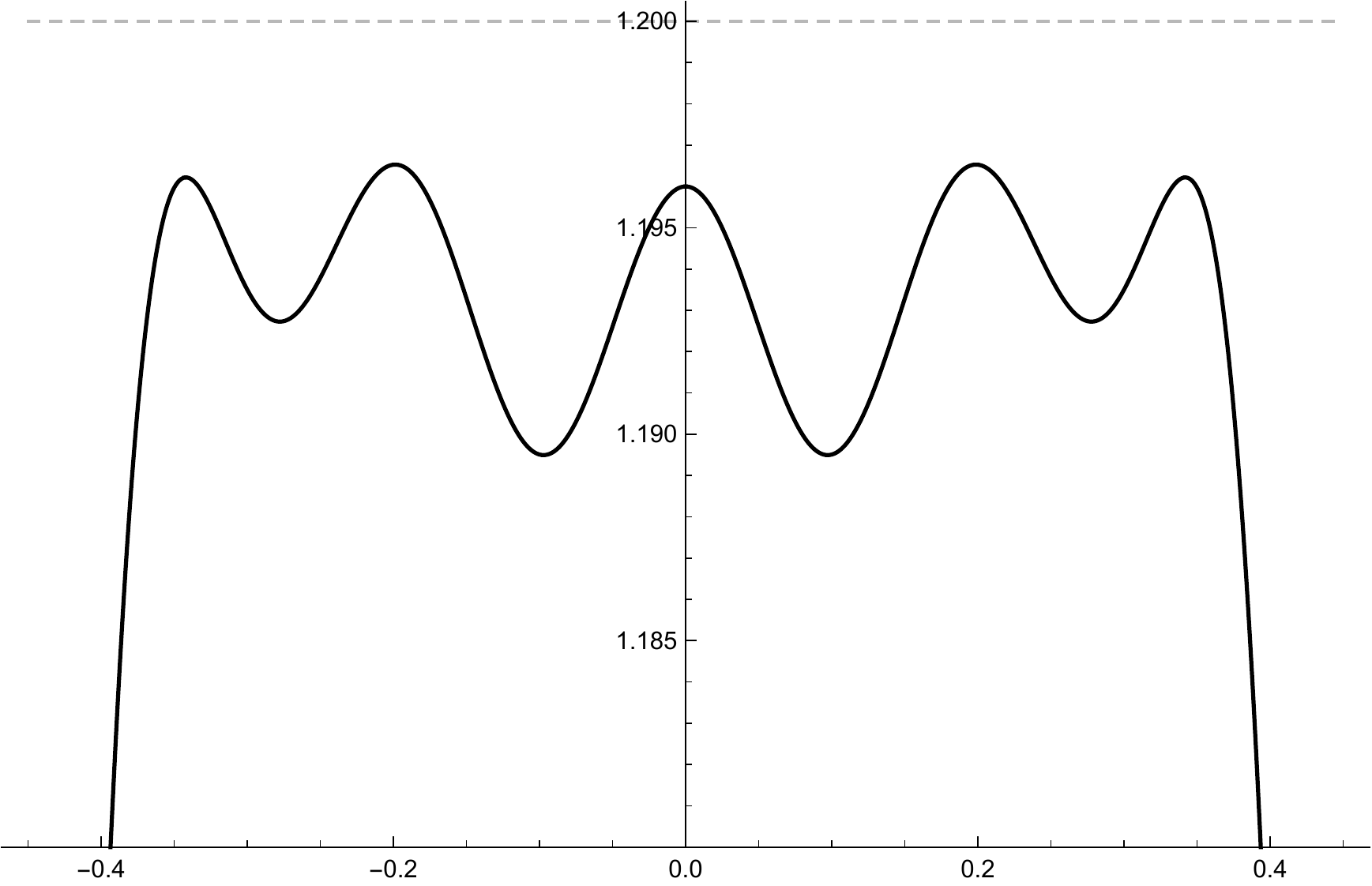} 
\caption{Graph of the function $\Psi$ in \eqref{June13_11:39am} in two different scales.}
\end{figure}

\subsection{Proof of Theorem \ref{Thm2} (c.2)} 
We have already seen that ${\mc C}^+(A) \leq {\mc C}^+(1)=2$. The other upper bound comes from the following dual formulation. Suppose that $\Phi \in L^{\infty}(\R)$ is a real-valued function such that its distributional Fourier transform is identically equal to $1$ on the interval $(-1,1)$ and $ - A \leq\widehat{\Phi}(t) - 1 \leq 0$ for all $t \in \R$. Then, for $F \in {\mc A}^+ \cap {\mc S}(\R)$ (as discussed in \S \ref{Sec_Approx}), we have
\begin{align*}
\|\Phi\|_{\infty} \int_{-\infty}^{\infty}|F(x)| \,\dx \geq \int_{-\infty}^{\infty} F(x)\,\Phi(x) \,\dx  = \int_{-\infty}^{\infty} \widehat{F}(t)\,\widehat{\Phi}(t) \,\dt  \geq F(0) - A\int_{[-1,1]^c} \big(\widehat{F}(t)\big)_+\,\dt\,,
\end{align*}
which leads to 
\begin{equation*}
 {\mc C}^{+}(A) \leq \|\Phi\|_{\infty}.
 \end{equation*}
The idea is to mollify the test function in \eqref{June13_11:39am} to bring down the delta functions to the required range, as done in \S \ref{Subs_UB}. Let $c = 0.444$ be the largest coefficient of a delta function in \eqref{June13_11:44am} and assume a priori that $A > 1 + \frac{c}{2}$ (so that our choice of $\lambda$ below is in fact positive). With the same notation as in \eqref{June09_4:47pm} we choose $\gamma -1 = \frac{\lambda}{2}$ and $\frac{c\gamma}{\lambda} = A-1$. Note that the four delta functions in \eqref{June13_11:44am} have negative coefficients, while the rest of the Fourier transform lies between $0$ and $1$, so we may take $A-1$ here instead of $A-2$. Moreover, since these delta functions are supported in non-consecutive integers, the condition $\gamma -1 = \frac{\lambda}{2}$ already guarantees that the mollified delta functions will not overlap (hence we do not need to assume here that $\lambda \leq \gamma$). This yields
$$\lambda = \frac{2}{ \frac{2}{c}(A-1) -1} \ \ \ {\rm and} \ \ \ \gamma =  \frac{1}{1 - \frac{c}{2(A-1)}}.$$
Since $\Phi(x) = \gamma\, \Psi(\gamma x)\,\widehat{R}(x/\lambda)$, our upper bound is $\|\Phi\|_{\infty} \leq \gamma \|\Psi\|_{\infty} < \gamma \times 1.2 $. 

\subsection{Proof of Theorem \ref{Thm2} (c.3)} For the specific value of  $A = \frac{36}{11}$, the lower bound described in \eqref{June01_3:17pm} and \eqref{May31_11:57pm} yields ${\mc C}^+\!(\tfrac{36}{11}) \geq 1.1569... $.  We found a better 
example through experimentation.  The function 
\begin{align}\label{June16_10:01am}
F(x) = -4.8\,x^2 e^{-3.3 x^2} + 1.5 \,x^2 e^{-7.4 x^2} + 520 \,x^{24} e^{-9.7 x^2} + 1.3\,e^{-2.8x^2} + 0.18 \,e^{-2x^2}
\end{align}
gives 
\begin{align*}
\frac{F(0) - A\int_{[-1,1]^c} \big(\widehat{F}(t)\big)_+\,\dt}{\|F\|_{1}}  = 1.1943... > \frac{25}{21}.
\end{align*}
We have found more complicated examples that do slightly better.

\begin{figure} \label{figure3}
\includegraphics[scale=.4]{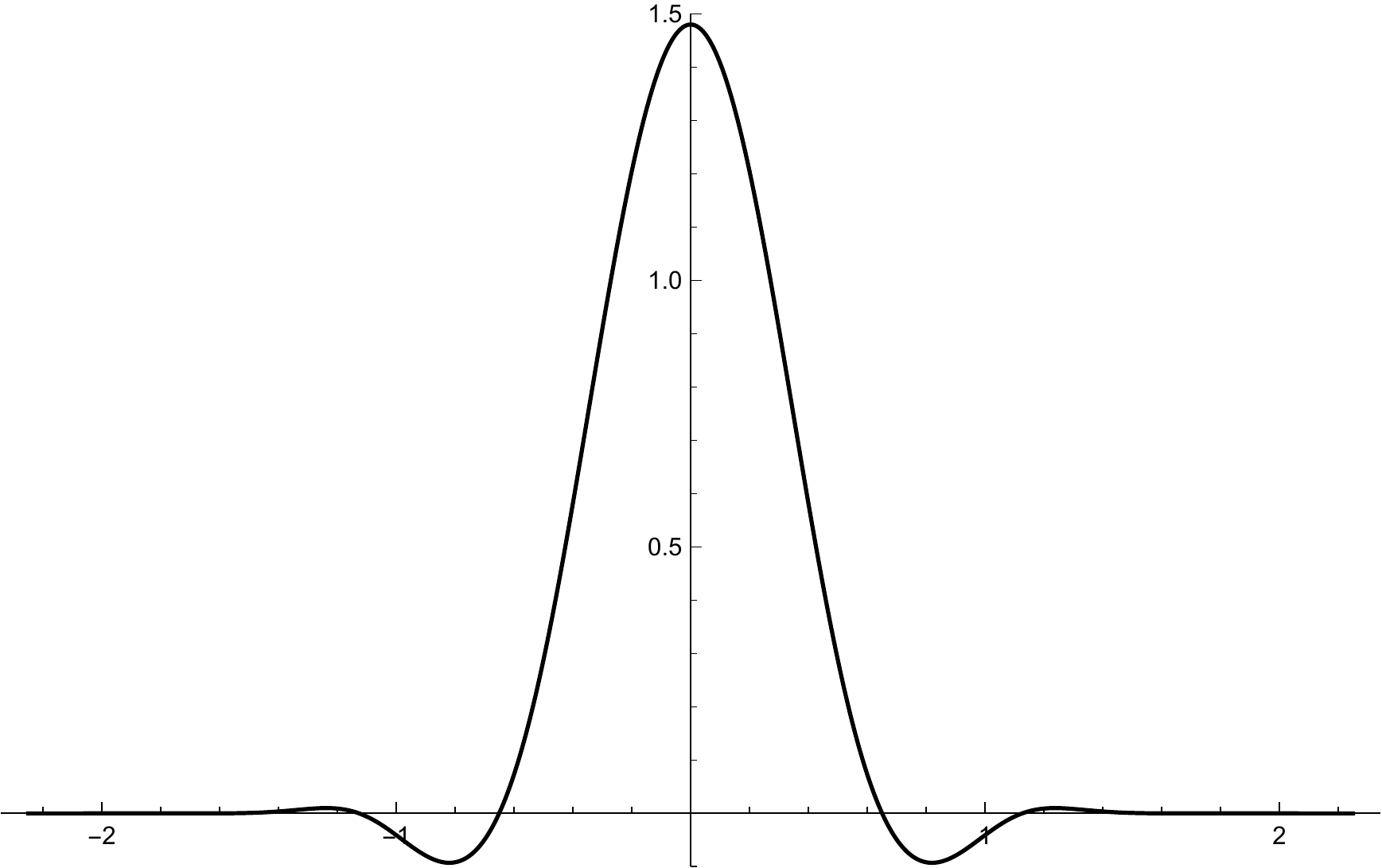} \qquad 
\includegraphics[scale=.4]{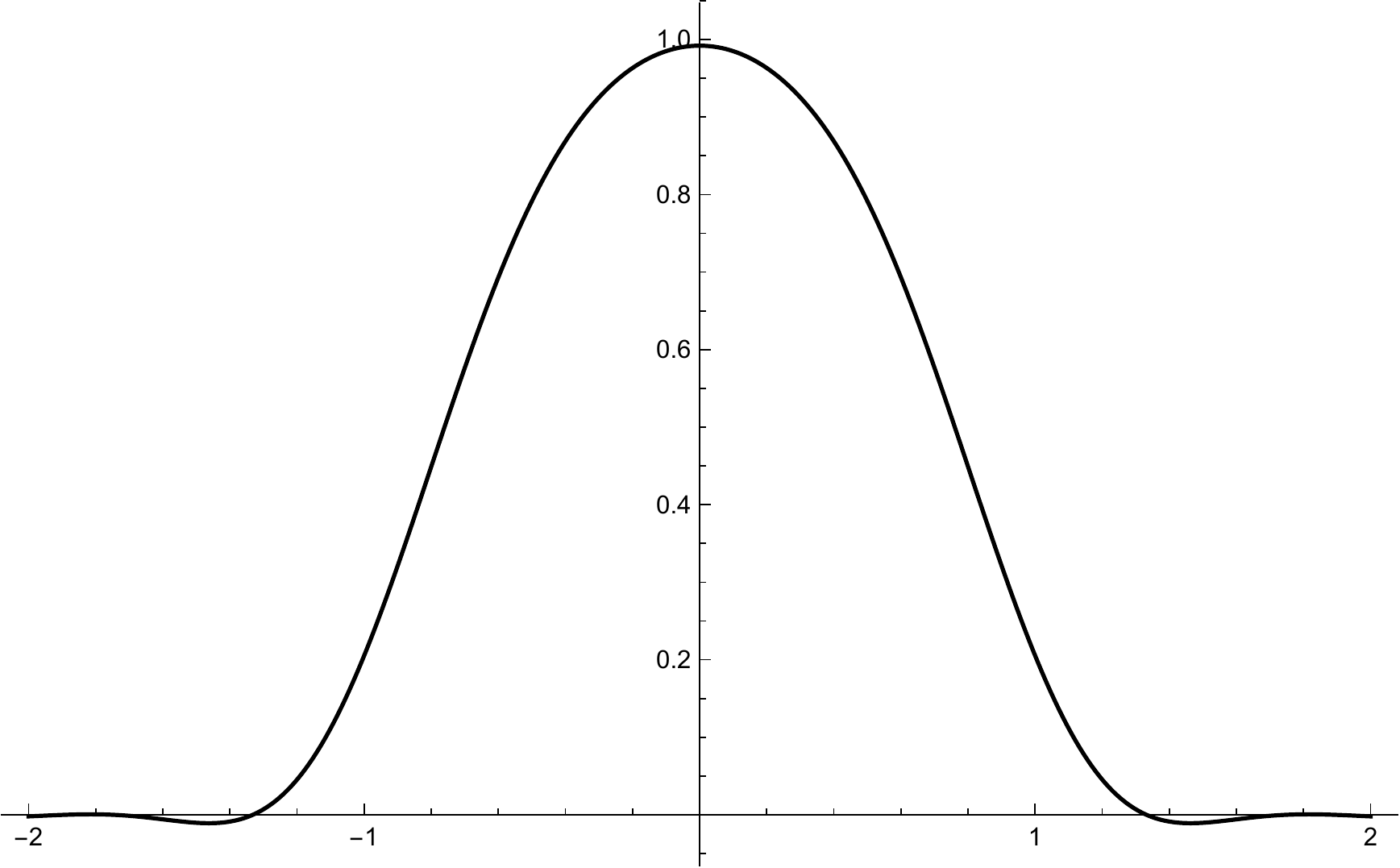} 
\caption{Graph of the function $F$ in \eqref{June16_10:01am} on the left, and graph of $\widehat{F}$ on the right.}
\end{figure}

\section{Prime gaps --- asymptotic version}\label{Sec3}

In this section we prove Theorem \ref{Thm3}. The proof uses two main tools: the explicit formula connecting the prime numbers and the zeros of the Riemann zeta-function, and the Brun-Titchmarsh inequality as expressed in \eqref{May31_2:05pm} and \eqref{May31_12:58pm}.

\begin{lemma}[Guinand-Weil explicit formula] \label{GW}
Let $h(s)$ be analytic in the strip $|\im{s}|\leq \tfrac12+\varepsilon$ for some $\varepsilon>0$, and assume that $|h(s)|\ll(1+|s|)^{-(1+\delta)}$ for some $\delta>0$ when $|\re{s}|\to\infty$. Then
	\begin{align*}
	\displaystyle\sum_{\rho}h\left(\frac{\rho - \tfrac12}{i}\right) & = h\left(\dfrac{1}{2i}\right)+h\left(-\dfrac{1}{2i}\right)-\dfrac{1}{2\pi}\widehat{h}(0)\log\pi+\dfrac{1}{2\pi}\int_{-\infty}^{\infty}h(u)\,\re{\dfrac{\Gamma'}{\Gamma}\!\left(\dfrac{1}{4}+\dfrac{iu}{2}\right)}\,\du \\
	 &  \ \ \ \ \ \ \ \ \ \ \ \ \ -\dfrac{1}{2\pi}\displaystyle\sum_{n\geq2}\dfrac{\Lambda(n)}{\sqrt{n}}\left(\widehat{h}\left(\dfrac{\log n}{2\pi}\right)+\widehat{h}\left(\dfrac{-\log n}{2\pi}\right)\right)\,, 
	\end{align*}
where $\rho = \beta + i \gamma$ are the non-trivial zeros of $\zeta(s)$, $\Gamma'/\Gamma$ is the logarithmic derivative of the Gamma function, and $\Lambda(n)$ is the Von-Mangoldt function defined to be $\log p$ if $n=p^m$ with $p$ a prime number and $m\geq 1$ an integer, and zero otherwise.
\end{lemma}
\begin{proof} The proof follows from \cite[Theorem 5.12]{IK}. 
\end{proof}
 
\subsection{Set-up}  Motivated by the discussion in \S \ref{Sec_Approx}, throughout this section we fix $F:\R \to \R$ to be an even and bandlimited Schwartz function, with $F(0) > 0$. Let us assume that $\supp(\widehat{F}) \subset [-N,N]$ for some parameter $N\ge 1$. It then follows that $F$ extends to an entire function (which we continue calling $F$) and the fact that $x^2F(x) \in L^{\infty}(\R)$ implies, via the Phragm\'{e}n-Lindel\"{o}f principle, that $|F(s)|\ll(1+|s|)^{-2}$ when $|\re{s}|\to\infty$. We may therefore apply the explicit formula (Lemma \ref{GW}).

\smallskip

Our idea to approach this problem can be summarized as follows. We use the explicit formula above to measure the size of an interval that does not contain too many primes. Note that the information about the primes is on the right-hand side of the formula, while on the left-hand side we have information on the zeros of $\zeta(s)$. We modify our test function $F$ in such a way that $\widehat{F}$ emphasizes the information on said interval, translating  and rescaling $\widehat{F}$ to concentrate the mass of $\widehat{F}$ on this interval. We then try to understand the effect of this localization in all the terms of the formula through an asymptotic analysis. 
Since the function $\widehat{F}$ must be small near its endpoints, it is advantageous to use the Brun-Titchmarsh inequality to (over) estimate the contribution from the primes on the edges of the interval.

\smallskip

Let $0< \Delta \leq 1$ and $1<a$ be free parameters to be   chosen later. We anticipate that we will be choosing $\Delta \to 0^+$ and $a \to \infty$, so it is not harmful to further assume that 
\begin{equation}\label{April28_11:18am}
2\pi\Delta N\leq \log a.
\end{equation}
Define $f(z) := \Delta F(\Delta z)$ and note that $\supp(\widehat{f}) \subset [-\Delta N,\Delta N]$. Assuming RH, an application of the explicit formula (Lemma \ref{GW}) to the entire function $h(z)= f(z)a^{iz}$ yields the following inequality:
\begin{align}\label{April21_11:23am}
\begin{split}
 f \left(\dfrac{1}{2i}\right) a^{1/2}+f\left(-\dfrac{1}{2i}\right) a^{-1/2} & \leq \sum_{\gamma} |f(\gamma)| + \dfrac{1}{2\pi}\widehat{f}\left(-\frac{\log a }{2\pi}\right)\log\pi \\
& \ \ \ \ \ \ \ + \left|\dfrac{1}{2\pi}\int_{-\infty}^{\infty}f(u)\,a^{iu}\,\re{\dfrac{\Gamma'}{\Gamma}\!\left(\dfrac{1}{4}+\dfrac{iu}{2}\right)}\,\du\right| \\
&   \ \ \ \ \ \ \ +\dfrac{1}{2\pi}\displaystyle\sum_{n\geq2}\dfrac{\Lambda(n)}{\sqrt{n}}\left(\widehat{f}\left(\dfrac{\log (n/a)}{2\pi}\right)_++\widehat{f}\left(-\dfrac{\log na}{2\pi}\right)_+\right).
\end{split}
\end{align}

\subsection{Proof of Theorem \ref{Thm3}} The idea is to proceed with an asymptotic evaluation of both sides of \eqref{April21_11:23am}. We start with its left-hand side. Note that 
\begin{align*}
\begin{split}
f \left(\dfrac{1}{2i}\right) & = \Delta F \left(\dfrac{\Delta}{2i}\right) = \Delta \int_{-N}^{N} e^{\pi t \Delta} \widehat{F}(t)\,\dt\\
& = \Delta \int_{-N}^{N} \widehat{F}(t)\,\dt + \Delta \int_{-N}^{N} \left(e^{\pi t \Delta} - 1\right) \widehat{F}(t)\,\dt\\
& = \Delta F(0) + O(\Delta^2).
\end{split}
\end{align*}
Therefore, the left-hand side of \eqref{April21_11:23am} equals
\begin{align*}
f \left(\dfrac{1}{2i}\right) a^{1/2}+f\left(-\dfrac{1}{2i}\right) a^{-1/2} = \Delta F(0) \big(a^{1/2} + a^{-1/2}\big) + O(\Delta^2 a^{1/2}). 
\end{align*}
For the right-hand side of \eqref{April21_11:23am}, we first consider the error terms. From \eqref{April28_11:18am} we have
\begin{align*}
\dfrac{1}{2\pi}\widehat{f}\left(-\frac{\log a }{2\pi}\right)\log\pi = \dfrac{1}{2\pi}\widehat{F}\left(-\frac{\log a }{2\pi\Delta}\right)\log\pi = 0.
\end{align*}
Also, using Stirling's formula $\frac{\Gamma'}{\Gamma}(s) = \log s + O(|s|^{-1})$ and \eqref{April28_11:18am}, we get
\begin{align*}
\begin{split}
\int_{-\infty}^{\infty}f(u)& \,a^{iu}\,\re{\dfrac{\Gamma'}{\Gamma}\!\left(\dfrac{1}{4}+\dfrac{iu}{2}\right)}\,\du = \int_{-\infty}^{\infty}F(y)\,e^{2 \pi i y (\frac{\log a}{2\pi \Delta})}\log\left|\dfrac{1}{4}+\dfrac{iy}{2\Delta}\right|\,\dy + O(1)\\
& =  \int_{-\infty}^{\infty}F(y)\,e^{2 \pi i y (\frac{\log a}{2\pi \Delta})}\left( \frac12\log( \Delta^2 + 4y^2) + \log\left(\dfrac{1}{4\Delta}\right)\right)\,\dy + O(1)\\
& = \log\left(\dfrac{1}{4\Delta}\right) \widehat{F}\left(-\frac{\log a }{2\pi\Delta}\right) + O(1) = O(1).
\end{split}
\end{align*}
Thus, we have deduced that
\begin{equation}\label{middle step}
\begin{split}
 \Delta F(0) \big(a^{1/2} + a^{-1/2}\big)  & \leq \sum_{\gamma} |f(\gamma)| + \dfrac{1}{2\pi}\displaystyle\sum_{n\geq2}\dfrac{\Lambda(n)}{\sqrt{n}}\left(\widehat{f}\left(\dfrac{\log (n/a)}{2\pi}\right)_++\widehat{f}\left(-\dfrac{\log na}{2\pi}\right)_+\right) \\
& \ \ \ \ \ \ \  + O(\Delta^2 a^{1/2}) + O(1). 
\end{split}
\end{equation}
It remains to estimate the two remaining sums on right-hand side of this inequality.

\subsubsection{The sum over zeros} Let $N(x)$ denote the number of zeros $\rho = \beta + i\gamma$ of $\zeta(s)$ with $0 < \gamma \leq x$. Using the fact that $N(x) = \frac{x}{2\pi} \log\frac{x}{2\pi} - \frac{x}{2\pi} + O(\log x)$, we evaluate the sum $\sum_{\gamma} |f(\gamma)|$ using summation by parts to get
$$\sum_{\gamma} |f(\gamma)| = \frac{1}{2 \pi} \int_{-\infty}^{\infty} |f(x)|\, \log^+\!\frac{|x|}{2\pi}\,\dx + O\big(\|f\|_{\infty}  + \| f'(x)\,\log^+\!|x|\|_1\big),$$
where $\log^+\!x = \max\{\log x,0\}$ for $x >0$. Recalling that $f(x) = \Delta F(\Delta x)$, this yields
\begin{align}\label{April20_4:49pm}
\begin{split}
\sum_{\gamma} |f(\gamma)| & = \frac{1}{2 \pi} \int_{-\infty}^{\infty} |F(y)|\, \log^+\!|y/2\pi\Delta|\,\dy + O(1) \\
& = \frac{\log(1/2\pi \Delta)}{2\pi} \|F\|_1 + O(1).
\end{split}
\end{align}

\subsubsection{The sum over primes and the choice of parameters} 
Fix $\alpha \geq0$ and assume that $c$ is a fixed positive constant such that 
$$\liminf_{x \to \infty} \frac{\pi\big(x + c \sqrt{x}\log x\big ) - \pi(x)}{\sqrt{x}} \leq  \alpha.$$
Then, given $\varepsilon >0$, there exists a sequence of $x \to \infty$ such that 
\begin{equation}\label{good_sequence}
\frac{\pi\big(x + c \sqrt{x}\log x\big ) - \pi(x)}{\sqrt{x}} \leq  \alpha + \varepsilon
\end{equation}
along this sequence. For each such $x$, we choose $a$ and $\Delta$ such that 
\begin{equation}\label{May01_11:47am}
[x, x + c \sqrt{x}\log x] = \left[a\,e^{-2 \pi \Delta}, a\, e^{2\pi\Delta}\right].
\end{equation}
Then (allowing the implicit constants in the big-$O$ notation here to depend on $c$) we have
\begin{equation}\label{Def_Delta_in_terms_of_x}
4\pi \Delta = \log\left(1 + c \frac{\log x}{\sqrt{x}}\right) = c \frac{\log x}{\sqrt{x}} + O\!\left(\frac{\log^2 x}{x}\right)
\end{equation}
and
\begin{equation}\label{Def_a_in_terms_of_x}
a = x \left( 1 + c \frac{\log x}{\sqrt{x}}\right)^{1/2} = x + O(\sqrt{x}\log x).
\end{equation}

\smallskip

By \eqref{April28_11:18am}, the sum we want to evaluate is
\begin{align}
\sum_{n\geq2}\dfrac{\Lambda(n)}{\sqrt{n}}\left(\widehat{f}\left(\dfrac{\log (n/a)}{2\pi}\right)_+ +\widehat{f}\left(-\dfrac{\log na}{2\pi}\right)_+\right) 
& = \sum_{n\geq2}\dfrac{\Lambda(n)}{\sqrt{n}}\widehat{F}\!\left(\dfrac{\log (n/a)}{2\pi\Delta}\right)_+.\label{April03_1:05pm}
\end{align}
Note that the last sum is supported on $n$ with $a\,e^{-2\pi \Delta N} \leq n \leq a\,e^{2\pi \Delta N}$. The contribution of the (at most) $(\alpha + \varepsilon) \sqrt{x}$ primes in the interval $(x, x + c \sqrt{x}\log x] = (a\,e^{-2\pi \Delta},a\,e^{2\pi \Delta}]$ to the sum \eqref{April03_1:05pm} is bounded above by (using the trivial bound $(\widehat{F}(t))_+ \leq \|F\|_1$)  
\begin{align*}
\leq \|F\|_1\sum_{p \in (a\,e^{-2\pi \Delta },a\,e^{2\pi \Delta }]}  \frac{\log p}{\sqrt{p}} & \leq  \|F\|_1 \,(\alpha + \varepsilon) \sqrt{x} \, \,\frac{\log x}{\sqrt{x}} =  \|F\|_1 \,(\alpha + \varepsilon)\,\log x.
\end{align*}
It is not hard to show that the contribution of the prime powers $n = p^k$ with $k \geq 2$ in the full interval $[a\,e^{-2\pi \Delta N},a\,e^{2\pi \Delta N}]$ to the sum \eqref{April03_1:05pm} is $O(1)$. It remains to estimate the contribution of the primes in the intervals $[a\,e^{-2\pi \Delta N},a\,e^{-2\pi \Delta}]$ and $[a\,e^{2\pi \Delta},a\,e^{2\pi \Delta N}]$, and for this we use  the Brun-Titchmarsh inequality. Let ${\bf B}$ be defined by \eqref{May31_2:05pm} and let ${\bf B'} > {\bf B}$. For $x$ sufficiently large we have 
\begin{align*}
\begin{split}
\sum_{1\leq |\frac{\log(p/a)}{2\pi\Delta}| \leq N}\frac{\log p}{\sqrt{p}}\,\widehat{F}\left(\dfrac{\log (p/a)}{2\pi\Delta}\right)_+ & \leq {\bf B'} \int_{1\leq |\frac{\log(t/a)}{2\pi\Delta}| \leq N}  \widehat{F}\left(\dfrac{\log (t/a)}{2\pi\Delta}\right)_+ \frac{\d t}{\sqrt{t}}\  + \ O(1)\\
& = {\bf B'}  \sqrt{a} \,(2\pi\Delta) \int_{[-1,1]^c}  \big(\widehat{F}(t)\big)_+\,\dt  \ +\  O(1).
\end{split}
\end{align*}
The inequality above can be seen by covering the intervals $[a\,e^{-2\pi \Delta N},a\,e^{-2\pi \Delta}]$ and $[a\,e^{2\pi \Delta},a\,e^{2\pi \Delta N}]$ by subintervals of size $\sqrt{a}$, and applying the Brun-Titchmarsh inequality in each summand of the corresponding Riemann-Stieltjes sum associated to this partition (the details of this argument are carried out in \S \ref{prime powers section} for a specific function and can be modified to handle the general case). Combining estimates, we see that 
\begin{equation}\label{sum bound}
\begin{split}
\sum_{n\geq2}\dfrac{\Lambda(n)}{\sqrt{n}}&\left(\widehat{f}\left(\dfrac{\log (n/a)}{2\pi}\right)_+ +\widehat{f}\left(-\dfrac{\log na}{2\pi}\right)_+\right) 
\\
& \qquad \qquad \le \|F\|_1 \,(\alpha + \varepsilon)\,\log x + {\bf B'}  \sqrt{a} \,(2\pi\Delta) \int_{[-1,1]^c}  \big(\widehat{F}(t)\big)_+\,\dt  \ +\  O(1).
\end{split}
\end{equation}

\subsubsection{Conclusion} Inserting the estimates in \eqref{April20_4:49pm} and \eqref{sum bound} into \eqref{middle step} and then rearranging terms, it follows that
\begin{align*}
\Delta \sqrt{a} \left( F(0) -  {\bf B'}  \int_{[-1,1]^c}  \big(\widehat{F}(t)\big)_+\,\dt \right) \leq \frac{\log(1/2\pi \Delta)}{2\pi} \|F\|_1 + \frac{1}{2\pi} \|F\|_1 \,(\alpha + \varepsilon)\,\log x + O(1)\,,
\end{align*}
where we have used \eqref{Def_Delta_in_terms_of_x} and \eqref{Def_a_in_terms_of_x} to combine the error terms. Sending $x \to \infty$ along the sequence \eqref{good_sequence}, we conclude that 
$$c \leq (1 + 2\alpha + 2 \varepsilon) \frac{\|F\|_1}{\left( F(0) -  {\bf B'}  \int_{[-1,1]^c}  \big(\widehat{F}(t)\big)_+\,\dt \right)}\,,$$
where we naturally assume that the denominator above is positive. Since this holds for all $\varepsilon >0$ and ${\bf B'} > {\bf B}$ we finally arrive at 
\begin{equation}\label{June14_2:23pm}
c \leq (1 + 2\alpha) \frac{\|F\|_1}{\left( F(0) -  {\bf B}  \int_{[-1,1]^c}  \big(\widehat{F}(t)\big)_+\,\dt \right)}.
\end{equation}
This is the connection to our extremal problem \eqref{Extremal_Problem_2} and the discussion in \S \ref{Sec_Approx} leads to the desired conclusion, since we may now optimize \eqref{June14_2:23pm} over such bandlimited $F$.

\section{Prime gaps --- explicit version}\label{Sec5}
We now move on to the proof of Theorem \ref{Thm5}. Instead of initially following the proof outlined in Section \ref{Sec3} with a particular choice of test function $F$ in the Guinand-Weil explicit formula (and carefully estimating the error terms), we start off slightly differently using a Mellin transform approach to the problem. For our fixed choice of test function, this approach simplifies some of our calculations. Moreover, it may be the case that the kernel we are using will be helpful in other applications. For a generic choice of test function, however, the Fourier transform approach to the problem used in the previous section is perhaps more illuminating. 

\begin{lemma}\label{Lem9}
Let $\vartheta$ and $\delta$ be positive numbers satisfying $\vartheta \delta = \pi/2$. Then, for $a> e^\delta$ and $\vartheta$ not an ordinate of a zero of $\zeta(s)$, we have
\begin{equation}\label{new explicit formula}
\begin{split}
\sum_{a e^{-\delta} \le n \le a e^{\delta}} \frac{\Lambda(n)}{\sqrt{n}} \cos \left( \vartheta \log \frac an\right) &=   \frac{\vartheta \sqrt{a}}{\frac{1}{4}+\vartheta^2} \big(e^{\delta/2} +e^{-\delta/2}\big) -  2 \vartheta \sum_{\gamma} \frac{ a^{i\gamma} \cos (\delta \gamma)}{\vartheta^2 -\gamma^2} 
\\
&\qquad - \sum_{n=1}^\infty \frac{\vartheta \, a^{-2n-1/2}}{(2n\!+\!\frac{1}{2})^2+\vartheta^2}\big(e^{(2n+1/2) \delta} +e^{-(2n+1/2)\delta}\big).
\end{split}
\end{equation}
Here the first sum on the right-hand side runs over the nontrivial zeros $\rho=1/2+i\gamma$ of $\zeta(s)$ where $\gamma \in \mathbb{C}$ with $|\mathrm{Re}(\gamma)|<1/2$.
\end{lemma}
\begin{proof} 
For any $c>0$, $\delta >0$ and $\xi >0$ we have
$$ 
\frac{1}{2\pi i} \int_{c - i\infty}^{c + i \infty} \xi^s \left( \frac{e^{\delta s} - e^{-\delta s}}{s} \right)\ds = \begin{cases} 
1, &\text{if  } e^{-\delta} <  \xi < e^{\delta}, \\ 
1/2,& \text{if  } \xi = e^{\pm\delta}, \\
0, &\text{otherwise}.
\end{cases} 
$$ 
It then follows, for any $c>1/2$, $a>0$, $\delta>0$ (assuming $a e^{\pm \delta} \not \in \mathbb{N}$), and any real number $\vartheta$, that 
$$ 
\frac{1}{2\pi i} \int_{c-i\infty}^{c+i\infty} -\frac{\zeta^{\prime}}{\zeta}(s+\tfrac12 + i\vartheta) \, a^{s+i\vartheta} \left( \frac{e^{\delta s}-e^{-\delta s}}{s} \right)  \ds 
\ = \sum_{a e^{-\delta} \le n \le a e^{\delta} } \frac{\Lambda(n)}{\sqrt{n}} \Big(\frac{a}{n} \Big)^{i\vartheta}\,.
$$ 
For details on this calculation we refer to \cite[Chapter 17]{Dav}. Applying this formula at $\vartheta$ and $-\vartheta$ and then adding, we deduce that
\begin{equation} \label{not simplified integral}
\begin{split}
2&\!\!\!\sum_{a e^{-\delta} \le n \le a e^{\delta}} \frac{\Lambda(n)}{\sqrt{n}} \cos \left( \vartheta \log \frac an\right) 
\\
&\quad= \frac{1}{2\pi i} \int_{c-i\infty}^{c+i\infty} -\frac{\zeta^{\prime}}{\zeta}(w+\tfrac 12) \, a^w \left( \frac{e^{\delta(w-i\vartheta)}\!-\!e^{-\delta(w-i\vartheta)}}{w-i\vartheta} + 
\frac{e^{\delta(w+i\vartheta)}\!-\!e^{-\delta(w+i\vartheta)}}{w+i\vartheta} \right) \, \dw.
\end{split}
\end{equation}
In the case $\vartheta \delta =\pi/2$, after dividing by 2, this formula simplifies to 
\begin{equation} 
\label{simplified integral}
\sum_{a e^{-\delta} \le n \le a e^{\delta}} \frac{\Lambda(n)}{\sqrt{n}} \cos \Big( \vartheta \log \frac an\Big) = \frac{1}{2\pi i} \int_{c-i\infty}^{c+i\infty} -\frac{\zeta^{\prime}}{\zeta}(w+\tfrac 12) \, a^w \frac{\vartheta}{w^2 +\vartheta^2} \big(e^{\delta w} + e^{-\delta w}\big) \, \dw ;
\end{equation} 
note the removable singularities of the integrand at $w=\pm i \vartheta$ and that the formula now holds when $a e^{\pm \delta} \in \mathbb{N}$, as well. Since $a> e^\delta$, we can shift the line of integration left from $\mathrm{Re}(w)=c$ to $\mathrm{Re}(w)=-\infty$ and, using the calculus of residues, the integral in \eqref{simplified integral} equals
\[
 \frac{\vartheta \sqrt{a}}{\frac{1}{4}+\vartheta^2} \big(e^{\delta/2} +e^{-\delta/2}\big) -  2 \vartheta \sum_{\gamma} \frac{ a^{i\gamma} \cos (\delta \gamma)}{\vartheta^2 -\gamma^2} 
 - \sum_{n=1}^\infty \frac{ \vartheta \, a^{-2n-1/2}}{(2n+\frac{1}{2})^2+\vartheta^2}\big(e^{(2n+1/2) \delta} +e^{-(2n+1/2)\delta}\big).
\]
Combining estimates, the lemma follows.
\end{proof}

\begin{remark}
Slightly more generally, if $\vartheta \delta \equiv \frac{\pi}{2} \! \pmod \pi$, then we can also evaluate the integral in \eqref{not simplified integral} in terms of an absolutely convergent sum over the nontrivial zeros of $\zeta(s)$ (but not otherwise).  
\end{remark}

\smallskip

Since $\vartheta \delta = \pi/2$, the first term on the right-hand side of  \eqref{new explicit formula} is
\[
\frac{\vartheta \sqrt{a}}{\frac{1}{4}+\vartheta^2} (e^{\delta/2}\!+\!e^{-\delta/2}) = 2 \pi \sqrt{a} \left( \frac{\delta \, (e^{\delta/2}\!+\! e^{-\delta/2})}{\pi^2 +\delta^2} \right) \ge 2 \pi \sqrt{a} \left( \frac{2 \delta}{\pi^2} \right) = \frac{4 \delta \sqrt{a}}{\pi}.
\]
Our assumptions below imply that $e^\delta/a \le 1/\sqrt{3}$, so the third term on the right-hand side of \eqref{new explicit formula} is bounded in absolute value by
\[
\frac{2}{\vartheta}\left(\frac{e^\delta}{a}\right)^{5/2} \sum_{n=0}^\infty \left(\frac{e^\delta}{a}\right)^{2n} \le \frac{2}{\vartheta}\left(\frac{e^\delta}{a}\right)^{5/2} \sum_{n=0}^\infty \left(\frac{1}{3}\right)^{n} = \frac{3}{\vartheta}\left(\frac{e^\delta}{a}\right)^{5/2}.
\]
Hence, taking absolute values in \eqref{new explicit formula} and using the previous two estimates, it follows that
\begin{equation} \label{before variable change}
\frac{4 \delta \sqrt{a}}{\pi} \ \le  \sum_{a e^{-\delta} \le n \le a e^{\delta}} \frac{\Lambda(n)}{\sqrt{n}} \cos \left( \vartheta \log \frac an\right) \, +\,  2 \vartheta \sum_{\gamma} \left|\frac{\cos(\delta \gamma)}{\vartheta^2-\gamma^2}\right| \, + \,   \frac{3}{\vartheta}\left(\frac{e^\delta}{a}\right)^{5/2}.
\end{equation}
At this point, it is convenient to make a change of variables so that we can retrace our steps from the proof of Theorem \ref{Thm3} in Section \ref{Sec3} using a dilation of the Fourier transform pair
\[
H(x) = \frac{\cos(2\pi x)}{1 - 16x^2} \quad  \text{and} \quad \widehat{H}(t) = \frac{\pi}{4} \cos\left(\frac{\pi t}{2}\right)\,\chi_{[-1, 1]}(t).
\]
We set $f(x)=\Delta F(\Delta x)$ where $F(x) = H(x/\lambda)$ so that $\widehat{F}(t) = \lambda \widehat{H}(\lambda t)$. Then, letting
\[
\delta = \frac{2\pi \Delta}{\lambda} \quad \text{and} \quad \vartheta = \frac{\lambda}{4 \Delta}
\]
in \eqref{before variable change} (note that $\vartheta \delta =\pi/2$), after a little rearranging it follows that
\begin{equation} \label{after variable change}
\Delta \sqrt{a} \ \le \  \sum_{\gamma} | f(\gamma)| \, +\,      \frac{1}{2\pi}\sum_{ n \ge 2 } \frac{\Lambda(n)}{\sqrt{n}} \widehat{F}\left( \frac{\log(n/a)}{2\pi\Delta} \right) \, + \,    \frac{3}{2} \Delta \left(\frac{e^{2\pi\Delta/\lambda}}{a}\right)^{5/2}.
\end{equation}
Note that the sum over $n$ is supported on the interval $\big(a e^{-2\pi \Delta/\lambda},a e^{2\pi \Delta/\lambda}\big)$. 

\smallskip

We assume that there are no primes in the interval $[x, x + c \sqrt{x}\log x]$ for $\frac{1}{2}\le c \le 1$, and we choose $a$ and $\Delta$ to satisfy \eqref{May01_11:47am}. In particular, the equalities in \eqref{Def_Delta_in_terms_of_x} and \eqref{Def_a_in_terms_of_x} still hold. As mentioned at the end of the introduction, we may assume that $x \geq 4 \cdot 10^{18}$. Using the fact that $\log(1 +y) \leq y$ for $y\ge 0$ in \eqref{Def_Delta_in_terms_of_x}, we note that
$\Delta \leq \frac{1}{4\pi} \left(\frac{\log x}{\sqrt{x}} \right) < 10^{-8}.$

\subsection{Sum over zeros}
We now explicitly estimate the sum over the zeros of the zeta function on the right-hand side of \eqref{after variable change}.
\begin{lemma}\label{Lem10}
Let $N(x)$ denote the number of zeros $\rho = \beta + i\gamma$ of $\zeta(s)$ with $0 < \gamma \leq x$. Then 
\begin{equation*}
\left|N(x) - \frac{x}{2\pi} \log\frac{x}{2\pi e} - \frac{7}{8} \right| \leq 0.15\log x + 3
\end{equation*}
for $x \geq e$.
\end{lemma}
\begin{proof}
The result holds for $e \leq x \leq 10$, since $N(10) = 0$. From  \cite[Corollary 1]{T} we have
\begin{equation}\label{April28_2:44pm}
\left|N(x) - \frac{x}{2\pi} \log\frac{x}{2\pi e} - \frac{7}{8} \right| \leq  0.112\log x + 0.278 \log \log x + 2.51 + \frac{0.2}{x},
\end{equation}
which holds for all $x \geq e$. The estimate 
\begin{equation}\label{April28_2:45pm}
0.278 \log \log x\leq 0.038\log x + 0.28
\end{equation}
holds for all $x \geq e$, while
\begin{equation}\label{April28_2:46pm}
\frac{0.2}{x} \leq 0.02
\end{equation}
holds for $x \geq 10$. Combining \eqref{April28_2:44pm}, \eqref{April28_2:45pm}, and \eqref{April28_2:46pm}, we arrive at our desired bound for $x \geq10$.
\end{proof}
Write
\begin{equation*}
N(x) = \frac{x}{2\pi} \log\frac{x}{2\pi e} + \frac{7}{8} + R(x)
\end{equation*}
and let $x_0 = 9.676...$ be such that 
$$
\frac{x_0}{2\pi} \log\frac{x_0}{2\pi e} + \frac{7}{8} = 0.
$$
Then, assuming the Riemann hypothesis and using summation by parts and Lemma \ref{Lem10}, we have
\begin{align*}
\sum_{\gamma>0} |f(\gamma)| & = \int_{x_0}^{\infty}  \left(\frac{1}{2\pi} \log\frac{x}{2\pi} \right) \,|f(x)|\,\dx- \int_{x_0}^{\infty}  R(x) \,|f|'(x)\,\dx\\
& \leq \int_{x_0}^{\infty}  \left(\frac{1}{2\pi} \log\frac{x}{2\pi} \right) \,|f(x)|\,\dx + \int_{x_0}^{\infty}  (0.15\log x  + 3)\,|f'(x)|\,\dx\\
& = \int_{\Delta x_0}^{\infty}  \left(\frac{1}{2\pi} \log\frac{y}{2\pi\Delta} \right) \,|F(y)|\,\dy + \Delta \int_{\Delta x_0}^{\infty}  (0.15\log \tfrac{y}{\Delta}  + 3)\,|F'(y)|\,\dy
\end{align*}
Therefore
\begin{align*}
\sum_{\gamma>0} |f(\gamma)|& \leq \frac{1}{2\pi} \int_{0}^{\infty} \log^+\! y  \,|F(y)|\,\dy + \frac{1}{2\pi}  \log(1/2 \pi \Delta) \int_{0}^{\infty}  |F(y)|\,\dy \\
 &  \ \ \ \ \ \ \ +  (0.15)\Delta \int_{0}^{\infty}  \log^+ y\,|F'(y)|\,\dy+  (0.15)\Delta \log(1/\Delta) \int_{0}^{\infty}|F'(y)|\,\dy \\
 & \ \ \ \ \ \ \  \ \ \ \ \ \ \ + 3\Delta\int_{0}^{\infty}|F'(y)|\,\dy.
 \end{align*}
 The same bound holds for the zeros with $\gamma <0$. Since $\Delta < 10^{-8}$ implies that $\Delta \log (1/\Delta) \leq 2 \times 10^{-7}$, we conclude that
 \begin{align}\label{May01_12:12pm}
 \begin{split}
\sum_{\gamma} |f(\gamma)| & \leq \frac{\log(1/2\pi \Delta)}{2\pi} \|F\|_1 + \frac{1}{2\pi} \| \log^+ \!|y|  \,F(y)\|_1 + (0.15) \times 10^{-8} \times \| \log^+\!|y|  \,F'(y)\|_1 \\
 &  \ \ \ \ \ \ \ \ \ \ \ \ \ \ \ + (3\times 10^{-8} + (0.15) \times 2 \times 10^{-7}) \|F'\|_1\\
& < \frac{\log(1/2\pi \Delta)}{2\pi} \|F\|_1+ \frac{0.070}{2\pi}.
\end{split}
\end{align}

\subsection{Sum over prime powers} \label{prime powers section}
We use a version of the Brun-Titchmarsh inequality due to Montgomery and Vaughan \cite[Theorem 2]{MV} which states that
\begin{equation}\label{Jun14_5:03pm}
\pi(x + y) - \pi(x) <  \frac{2y}{\log y},
\end{equation}
for all $x, y >1$.  For us, the relevant range is $y \ge \sqrt{x}$, so that \eqref{Jun14_5:03pm} corresponds to 
an application of the Brun-Titchmarsh inequality with the bound ${\bf B} \le 4$. This is slightly worse than Iwaniec's bound \eqref{May31_12:58pm} but is completely explicit. With $A=4$,  the lower bound in \eqref{May31_11:57pm} was established in \S \ref{Jun14_2:56pm} with a dilation of the function $H(x)$ with dilation parameter $\lambda=\lambda(4) = 0.892422....$, leading to the bound ${\mc C}(4) \geq 1.141186... = (0.8762...)^{-1}$. For the sake of simplicity, we work instead with the dilation parameter $\lambda = 0.9$ and note that for $F(x) = H(x/\lambda)$ we have
\begin{align}\label{June15_12:00pm}
J(F) = \frac{F(0) - 4\int_{[-1,1]^c} \big|\widehat{F}(t)\big|\,\dt}{\|F\|_{1}} = 1.1405... > \frac{25}{22}.
\end{align}

With $a$ and $\Delta$ chosen as in \eqref{May01_11:47am}, we need to estimate the contribution of the primes $p$ such that $1 < |\frac{\log p/a}{2\pi \Delta}| \leq \lambda^{-1}$  to the sum over $n$ in \eqref{after variable change}. We cover the interval $(a\,e^{2\pi \Delta},a\,e^{2\pi \Delta \lambda^{-1}}] \subset \cup_{j=0}^{J-1}(x_j, x_{j+1}]$, with $x_0 = a\,e^{2\pi \Delta}$ and $x_{j+1} = x_j + \sqrt{x_j}$. Using \eqref{Jun14_5:03pm} in each subinterval $(x_j, x_{j+1}]$ and the fact that $\widehat{F}$ is decreasing on $[0,\lambda^{-1}]$ we obtain
\begin{align}
\allowdisplaybreaks
\sum_{1 < \frac{\log p/a}{2\pi \Delta}  \leq \lambda^{-1}}& \frac{\log p}{\sqrt{p}} \, \widehat{F}\left( \frac{\log(p/a)}{2\pi \Delta} \right)  \leq \sum_{j=0}^{J-1} \left(\frac{\log x_j}{\sqrt{x_j}} \, \widehat{F}\left( \frac{\log (x_j/a)}{2\pi \Delta} \right) \right)\frac{4 \sqrt{x_j}}{\log x_j} \nonumber \\
& \leq 4 \widehat{F}(1) + \frac{4}{\sqrt{a}} \sum_{j=1}^{J-1}  \widehat{F}\left( \frac{\log (x_j/a)}{2\pi \Delta}\right) \sqrt{x_{j-1}} \nonumber  \\
&  \leq 4 \widehat{F}(1) +  \frac{4}{\sqrt{a}} \int_{x_0 }^{x_J}\widehat{F}\left( \frac{\log (t/a)}{2\pi \Delta} \right) \,\dt \label{June15_11:46am} \\
& = 4 \widehat{F}(1)  + 4 \sqrt{a} \,(2\pi \Delta) \int_{1}^{\lambda^{-1}} \widehat{F}(y) \, e^{2\pi \Delta y}\,\dy  \nonumber \\
& =  4 \widehat{F}(1)  + 4 \sqrt{a} \,(2\pi \Delta) \int_{1}^{\lambda^{-1}} \widehat{F}(y) \,\dy + 4 \sqrt{a} \,(2\pi \Delta) \int_{1}^{\lambda^{-1}} \widehat{F}(y) \, \left(e^{2\pi \Delta y} - 1\right)\,\dy \nonumber \\
& \leq 4 \widehat{F}(1)  + 4 \sqrt{a} \,(2\pi \Delta)\left( \int_{1}^{\lambda^{-1}} \widehat{F}(y) \,\dy \right) + 4 \sqrt{a} \,(4\pi \Delta)^2\widehat{F}(1) (\lambda^{-1} - 1), \nonumber 
\end{align}
where we have used the basic estimate $e^{x}-1 \leq 2x$, for $x \leq 1$, in the last passage. We treat the other interval in a similar way, covering $[a\,e^{-2\pi \Delta \lambda^{-1}},a\,e^{-2\pi \Delta }) \subset \cup_{j=0}^{L-1}[x_{j+1}, x_{j})$, with $x_0 = a\,e^{-2\pi \Delta}$ and $x_{j} = x_{j+1} + \sqrt{x_{j+1}}$. Using \eqref{Jun14_5:03pm} in each subinterval $[x_{j+1}, x_{j})$ and the fact that $\widehat{F}$ is increasing on $[-\lambda^{-1},0]$ we obtain
\begin{align}
 \sum_{-\lambda^{-1}< \frac{\log p/a}{2\pi \Delta}  \leq -1}& \frac{\log p}{\sqrt{p}} \, \widehat{F}\left( \frac{\log(p/a)}{2\pi \Delta} \right)  \leq \sum_{j=0}^{L-1} \left(\frac{\log x_{j+1}}{\sqrt{x_{j+1}}} \, \widehat{F}\left( \frac{\log (x_j/a)}{2\pi \Delta} \right) \right)\frac{4 \sqrt{x_{j+1}}}{\log x_{j+1}} \nonumber \\
& \leq 4 \widehat{F}(-1) + \frac{4}{\sqrt{a/(e^{4\pi\Delta})}} \sum_{j=1}^{L-1}  \widehat{F}\left( \frac{\log (x_j/a)}{2\pi \Delta}\right) \sqrt{x_{j}}  \nonumber \\
&  \leq 4 \widehat{F}(-1) +  \frac{4e^{2\pi\Delta}}{\sqrt{a}} \int_{x_L }^{x_0}\widehat{F}\left( \frac{\log (t/a)}{2\pi \Delta} \right) \,\dt  \label{June15_11:47am}\\
& \leq 4 \widehat{F}(-1)  + 4 \sqrt{a} \,e^{2\pi\Delta}\,(2\pi \Delta) \int_{-\lambda^{-1}}^{-1} \widehat{F}(y)\, \dy \nonumber \\
& \leq 4 \widehat{F}(-1)  + 4 \sqrt{a} \,(2\pi \Delta) \int_{-\lambda^{-1}}^{-1} \widehat{F}(y)\, \dy\ + 8 \sqrt{a} \,(2\pi \Delta)^2\,\widehat{F}(-1) (\lambda^{-1} - 1).\nonumber 
\end{align}
Combining \eqref{June15_11:46am} and \eqref{June15_11:47am} we conclude that
\begin{align}\label{June15_11:54am}
\begin{split}
 \sum_{1 < |\frac{\log p/a}{2\pi \Delta} | \leq \lambda^{-1}}& \frac{\log p}{\sqrt{p}} \, \widehat{F}\left( \frac{\log(p/a)}{2\pi \Delta} \right)  \\
 & \leq 8 \widehat{F}(1)  + 4  \sqrt{a} \,(2\pi \Delta) \int_{[-1,1]^c} \widehat{F}(y)\, \dy + 24\sqrt{a} \,(2\pi \Delta)^2\,\widehat{F}(1) (\lambda^{-1} - 1)\\
 & \leq 0.886 + 4  \sqrt{a} \,(2\pi \Delta) \int_{[-1,1]^c} \widehat{F}(y)\, \dy. 
 \end{split}
 \end{align}
Here we have used the estimate $8 \widehat{F}(1) \le 0.885$ along with the inequalities $a\le 4x$ and $\log(1+y) \le y$ for $y\ge 0$ in \eqref{Def_Delta_in_terms_of_x} and \eqref{Def_a_in_terms_of_x} to see that $\sqrt{a} \,(2\pi \Delta)^2 \le \frac{c^2 \log^2x}{\sqrt{x}} \le 10^{-7}$ for $c\le 1$ and $x \ge 4\cdot 10^{18}$ and thus that $24\sqrt{a} \,(2\pi \Delta)^2\,\widehat{F}(1) (\lambda^{-1} - 1) \le 0.001$.

\smallskip
 
Since $\supp(\widehat{F}) \subset [-\lambda^{-1}, \lambda^{-1}] \subset [-2,2]$  and 
$|{\widehat F}(y)|\le \pi \lambda/4<1$, the contribution from the prime powers $n=p^k$ with $k\ge 2$ to the sum over $n$ in \eqref{after variable change} is
$$ 
\le \ \sum_{k \ge 2} \sum_{\substack{ ae^{-4\pi \Delta} \le n  \le  ae^{4\pi \Delta} \\ n=p^k} } \frac{\log p}{\sqrt{n}}  
\ \le \ \frac{\log (ae^{-4\pi\Delta})}{2\sqrt{ae^{-4\pi \Delta}}} \!\!\!\!\!\! \sum_{\substack{ k\ge 2 \\ k\le \log (ae^{4\pi \Delta})/\log 2}} 
\!\!\!\!\!\!  \Big(1 + a^{\frac 1k} (e^{4\pi \Delta/k} - e^{-4\pi \Delta/k}) \Big), 
$$
where we used a trivial estimate for the total number of $k$th powers that can lie in the interval $[ae^{-4\pi \Delta}, ae^{4\pi \Delta}]$.  This is readily bounded by 
\begin{align} 
&\le \frac{\log (ae^{-4\pi \Delta})}{2\sqrt{ae^{-4\pi \Delta}} } \Big (2 \sqrt{a} \big(e^{2\pi \Delta }-e^{-2\pi \Delta}\big) \big)\frac{\log (ae^{4\pi \Delta})}{\log 2}  = \frac{\log (ae^{-4\pi \Delta})\log (ae^{4\pi \Delta})}{a\, e^{-2\pi\Delta}\,\log 2} \big(a e^{2\pi \Delta }-a e^{-2\pi \Delta} \big) \nonumber \\
& =  \frac{\log (ae^{-4\pi \Delta})\log (ae^{4\pi \Delta})}{a \,e^{-2\pi\Delta}\,\log 2} \, c \sqrt{x} \log x < \frac{2 (\log a +1)^3}{\log 2\,\sqrt{a}}  < 0.001. \label{May01:12:13pm}
\end{align}

\subsection{Finishing the proof} Note that $a e^{-2\pi\Delta/\lambda} \ge x/2 \ge 2 \cdot 10^{18}$ and thus
\[
\frac{3}{2} \Delta \left(\frac{e^{2\pi\Delta/\lambda}}{a}\right)^{5/2} \le \, \frac{0.001}{2\pi}.
\]
Combining this estimate with \eqref{after variable change}, \eqref{May01_12:12pm}, \eqref{June15_11:54am}, and \eqref{May01:12:13pm}, (after multiplying both sides by $2\pi$) we derive that 
\begin{align*}
\sqrt{a}\,(2\pi \Delta )&  \leq \log\left(\frac{1}{2\pi \Delta}\right) \|F\|_1+ 4  \sqrt{a} \,(2\pi \Delta) \int_{[-1,1]^c} \widehat{F}(y)\, \dy + 0.958.
\end{align*}
Rearranging and dividing by $\|F\|_1 = \lambda \|H\|_1 = 0.83337\ldots$ we obtain (with $J(F)$ defined in \eqref{June15_12:00pm})
\begin{align}\label{June15:12:09pm}
J(F)\,\sqrt{a}\,(2\pi \Delta) \leq \log\left(\frac{1}{2\pi \Delta}\right) + 1.16.
\end{align}
Since $a \ge x$, $1 \ge c \ge \frac{1}{2}$, and $\log(1+y) \ge y -\frac{y^2}{2}$ for $y\ge 0$, we derive from \eqref{Def_Delta_in_terms_of_x} and \eqref{Def_a_in_terms_of_x} that
\[
\sqrt{a}\,(2\pi \Delta) \ge \frac{c}{2} \log x - \frac{c^2 \log^2 x}{4\sqrt{x}} \ge \frac{c}{2} \log x - 0.001.
\]
Using the inequalities $c\ge \frac{1}{2}$ and $\log(1+y) \ge y \log 2$, which holds for $0\le y \le 1$, it follows that
$2\pi \Delta = \frac{1}{2} \log\left(1 + c \frac{\log x}{\sqrt{x}}\right) \ge \frac{c\log 2}{2} \frac{\log x}{\sqrt{x}}$ 
and therefore
\[
\begin{split}
\log\left(\frac{1}{2\pi\Delta}\right) &\le \log\left( \frac{2}{c\log 2} \frac{\sqrt{x}}{\log x} \right) 
\\
&\le \frac{1}{2}\log x - \log\left( \frac{\log 2}{4} \log( 4\cdot 10^{18}) \right) \le \frac{1}{2} \log x -2.
\end{split}
\]
Inserting these estimates into \eqref{June15:12:09pm}, we derive that
\[
\frac{c \, J(F)}{2} \log x \le \frac{1}{2} \log x - \frac{1}{2}.
\]
This is not possible if $c = \frac{1}{J(F)} < \frac{22}{25}$. Hence there must be a prime in the interval $[x,x+ \frac{22}{25}\sqrt{x} \log x]$. 
\section{Concluding remarks}\label{Sec7}
There are several related extremal problems in Fourier analysis that could be the sources of further investigation. We briefly discuss a few of these here.

\subsection{Multidimensional analogues} The corresponding versions of the extremal problems \eqref{Extremal_Problem_1} -- \eqref{Extremal_Problem_2_infty} in $\R^d$ arise as natural generalizations. The compact interval $[-1,1] \subset\R$ could be replaced by any convex, compact, and symmetric set $K \subset \R^d$, for instance. Of those, the most basic ones are certainly the cube $Q = [-1,1]^d$ and the unit Euclidean ball $B = \{x \in \R^d; |x|\leq1\}$. The same ideas used here could be applied to show the existence of extremizers in this general situation. By averaging over the group of symmetries of $K$, one can show that extremizers admit, without loss of generality, these symmetries. Note that a crucial step in our proof of the uniqueness of extremizers in Section \ref{Sec3_Uniqueness} (for the bandlimited problem \eqref{Extremal_Problem_1_infty}) was the ability to write a nonnegative function with Fourier transform supported in $2K$ as the square of a function whose Fourier transform is supported in $K$. In general, this decomposition is not available for any given $K$, but in the case of the unit ball $B$, with respect to radial functions, this statement holds. This was proved, for instance, in \cite{CL2,HV}, exploring the connection with the theory of Hilbert spaces of entire functions of L. de Branges. Hence, in dimension $d \geq1$ and for $K = B$, one has indeed the uniqueness of radial extremizers (up to multiplication by a complex scalar) for the multidimensional version of \eqref{Extremal_Problem_1_infty}. Letting ${\mc C}_{d,K}(A)$ denote the sharp constant in the multidimensional version of \eqref{Extremal_Problem_1} -- \eqref{Extremal_Problem_1_infty}, one can show that ${\mc C}_{d,Q}(\infty) = {\mc C}(\infty)^d$, and a tensor product of one-dimensional extremizers is an extremizer for the multivariable problem. In the general case, one has ${\mc C}_{d,K}(\infty) \leq {\rm vol}(K)$. A lower bound for ${\mc C}_{d,K}(\infty)$ may come, for instance, from the solution of the ``one-delta problem for $K$", which is the same problem as \eqref{Extremal_Problem_1_infty} with the additional constraint that $F \geq 0$. Such problem is also vastly open, having been solved only in a few particular cases such as the cube $Q$ and the ball $B$ (see the discussion in \cite{BK, GKM,K}). It would be interesting to have refined upper and lower bounds for all of these extremal problems, as we have here in our Theorems \ref{Thm1} and \ref{Thm2}.

\subsection{Sphere packing} The following extremal problem in Fourier analysis was proposed by Cohn and Elkies \cite{CE} in connection to the sphere packing problem. Find
\begin{equation}\label{June15_5:21pm}
C = \sup_{\substack{F \in {\mc E}^+_d \\ F \neq 0}} \frac{F(0)}{\widehat{F}(0)}\,,
\end{equation}
where the supremum is taken over the class ${\mc E}^+_d$ of real-valued, continuous, and integrable functions $F:\R^d \to \R$ with $F\geq 0$ and $\widehat{F}(y) \leq0$ for $|y| \geq 1$. This is the multidimensional analogue of our extremal problem \eqref{Extremal_Problem_2_infty} with the additional constraint that $F \geq0$. By averaging over the group of rotations $SO(d)$ we may restrict the search to radial functions and by following the outline of \S\ref{June08_11:01am} and \S\ref{June16_8:25am} we obtain the next result.\footnote{This result has been previously communicated by E. Carneiro and Alvaro A. Gomez (with a slightly different proof than the one presented here), as part of the M.Sc. thesis of the latter under the supervision of the former.}
\begin{proposition}
There exists a radial extremizer for \eqref{June15_5:21pm}.
\end{proposition}
As a matter of fact, Cohn and Elkies \cite{CE} proposed this optimization problem over the more restrictive class of admissible functions $F:\R^d \to \R$ such that $|F|$ and $|\widehat{F}|$ are bounded above by constant times $(1 + |x|)^{-d-\delta}$ for some $\delta>0$. Standard approximation arguments show that the sharp constant over this restricted class is the same $C$ in \eqref{June15_5:21pm}, although extremizers of \eqref{June15_5:21pm}, in principle, need not have this particular decay. In addition to dimension $d=1$, the value of the sharp constant in \eqref{June15_5:21pm} is known only in dimensions $d=8$ and $24$ (see \cite{Vi} and \cite{CKMRV}, respectively). The extremizers found by Viazovska in \cite{Vi} and by Cohn, Kumar, Miller, Radchenko, and Viazovska in \cite{CKMRV} are indeed radial Schwartz functions.

\section*{Acknowledgments}
The authors are thankful to Andr\'{e}s Chirre for very helpful discussions during the preparation of this work, to Dan Goldston for encouragement, and to Dimitar Dimitrov for bringing references \cite{AKP, Gor, Gor2, Tai} to our attention. E.C. acknowledges support from CNPq-Brazil, FAPERJ-Brazil, and the Fulbright Junior Faculty Award, and is also thankful to Stanford University for the support and warm hospitality. M.B.M. was supported in part by NSA Young Investigator Grants H98230-15-1-0231 and H98230-16-1-0311, and thanks Stanford University for hosting him on two research visits.  K.S. was
partially supported by NSF grant DMS 1500237, and a Simons Investigator grant from the Simons Foundation.

\end{document}